\documentclass[11pt]{article}

   \renewcommand{\footnote}[1]{
  \textsuperscript{ %ecriture en exposant
    \addtocounter{footnote}{1} %incrementation du compteur
    (\thefootnote) % impression au format "(compteur)"
  }
   \footnotetext{#1} % la note de bas de page
}

\usepackage{bm}
\usepackage{amsmath}
\usepackage{amsthm}
\usepackage{trfsigns}
\usepackage{wasysym}
\usepackage{amssymb}
\usepackage{amsfonts}
\usepackage{graphicx}  
\usepackage[frenchb,english]{babel}
\usepackage[utf8]{inputenc}%a la place d'applemac utf8
\usepackage[T1]{fontenc}
\usepackage{mathrsfs}
\usepackage{pdfsync}
\usepackage{enumerate}
\usepackage{version}
\usepackage{calc}
\usepackage{subfigure}

 \usepackage{pstricks,pstricks-add,pst-math,pst-xkey}
 \usepackage{float}
 \usepackage{indentfirst}
 \usepackage{minitoc}
 \usepackage[title,titletoc]{appendix}

\newtheorem{prop}{Proposition}%[chapter]

\newtheorem{thm}{Theorem}%[chapter]
\newtheorem{lem}{Lemma}%[chapter]
\newtheorem{cor}{Corollary}%[chapter]
\newtheorem{lemimnoi}{Lemma}

\newtheorem{remark}{Remark}%[chapter]

\title{ LOCAL MINIMIZERS OF THE GINZBURG-LANDAU FUNCTIONAL WITH PRESCRIBED DEGREES}
\author{Mickaël~Dos\,Santos \footnote{Université de Lyon
CNRS UMR 5208
INSA de Lyon
Institut Camille Jordan
20, avenue Albert Einstein
69621 Villeurbanne Cedex
France}
\\{\tt dossantos@math.univ-lyon1.fr}
}

\numberwithin{equation}{section}

\def\v{\varepsilon}
\def\R{\mathbb{R}}
\def\Z{\mathbb{Z}}
\def\N{\mathbb{N}}
\def\C{\mathbb{C}}
\def\dom{\mathcal{D}}
\def\O{\Omega}

\def\n{\nabla}
\def\p{\partial}
\def\*{u_{*}}

\def\be{\begin{eqnarray}}
\def\ee{\end{eqnarray}}
\def\no{\nonumber}

\def\1{\textrm{1\kern-0.25emI}}
\def\deg{{\rm deg}}
\def\dist{{\rm dist}}
\def\tr{{\rm tr}}
\def\Div{{\rm div}}

\def\H{\mathcal{H}}
\def\mint_#1{\mathchoice 
{\mathop{\vrule width 6pt height 3 pt depth -2.5pt 
\kern -8.8pt \intop}\nolimits_{#1}}% 
{\mathop{\vrule width 5pt height 3 pt depth -2.6pt 
\kern -6.5pt \intop}\nolimits_{#1}}% 
{\mathop{\vrule width 5pt height 3 pt depth -2.6pt 
\kern -6pt \intop}\nolimits_{#1}}% 
{\mathop{\vrule width 5pt height 3 pt depth -2.6pt 
\kern -6pt \intop}\nolimits_{#1}}}  
\def\d{\displaystyle}

\def\Eeps(u){E_\v(u):=\frac{1}{2}\int_A{|\n u|^2\di x}+\frac{1}{4\v^2}\int_A{\left(1-|u|^2\right)^2\di x}}
\def\k{\kappa}

\def\F{\mathcal{F}}
\def\J{\mathcal{J}}

\def\ab{{\rm abdeg}}
\def\di{\,\text{d}}

\def\db{\mathbf{d}}
\def\tw{\tilde{w}}

\def\faible{\rightharpoonup}

\def\Jp{\J_{\bfp,q}^\db}
\def\bfp{{\bf p}}
\def\tae{\textsf{\ae}}
\def\o{\omega}

\def\goto{\rightarrow}

\begin{document}
\maketitle
\begin{abstract}
We consider, in a smooth bounded multiply connected domain $\dom\subset\R^2$, the Ginzburg-Landau energy $E_\v(u)=\frac{1}{2}\int_\dom{\left\{|\n u|^2+\frac{1}{2\v^2}(1-|u|^2)^2\right\}}$ subject to prescribed degree conditions on each component of $\p\dom$. In general, minimal energy maps do not exist \cite{BeMi1}. When $\dom$ has a single hole, Berlyand and Rybalko \cite{BeRy1} proved that for small $\v$ local minimizers do exist. We extend the result in \cite{BeRy1}: $\d E_\v(u)$ has, in domains $\dom$ with $2,3,...$ holes and for small $\v$, local minimizers. Our approach is very similar to the one in \cite{BeRy1}; the main difference stems in the construction of test functions with energy control. 
\end{abstract}
%\minitoc
%\newpage
 
\tableofcontents
\newcounter{3.eps}%[3.eps] 
\newcommand\etiquette[2]{\refstepcounter{#1}\arabic{#1}\label{#2}}
\section{Introduction}\label{S1.IntroductionLocalMin}

This article deals with the existence problem of local minimizers of the %(without magnetic field) 
Ginzburg-Landau functional with prescribed degrees in a 2D perforated domain $\dom$. 

The domain we consider is of the form $\dom=\O\setminus\cup_{i\in\N_N}\overline{\o}_i$, where $N\in\N^*$, $\O$ and the $\o_i$'s are simply connected, bounded and smooth open sets of $\R^2$. 

We assume that $\overline{\o_i}\subset\O$ and $\overline{\o_i}\cap\overline{\o_j}=\emptyset$ for $i,j\in\N_N:=\{1,...,N\},i \neq j$. %The case of $N=1$ was treated in \cite{BeRy1}.

%Our main result is the existence for all $((p_1,...,p_N),q)\in\Z^{N+1}$ and for $\v>0$ small enough of solution of 
%\begin{equation}\label{3.8}
%-\Delta u=\frac{1}{\v^2}u(1-|u|^2)\text{ in }A,
%\end{equation}
%\begin{equation}\label{3.8'}
%|u|=1\text{ et }u\times\p_\nu u=0\text{ on }\p A,
%\end{equation}
%\begin{equation}\label{3.8deg}
%\deg_{\p\omega_i}(u)=p_i,\,\deg_{\p\O}(u)=q.
%\end{equation}
%Where $u:A\rightarrow\C\simeq\R^2$ and for $\gamma$ a Jordan curve, $\deg_\gamma u$ is the topological degree of $u$ on $\gamma$.

%(\ref{3.8}) is the Euler-Lagrange PDE associate to the
The Ginzburg-Landau functional is
\begin{equation}\label{3.1}
E_\v(u,\dom):=\frac{1}{2}\int_\dom{\left\{|\n u|^2+\frac{1}{2\v^2}\left(1-|u|^2\right)^2\right\}\di x}
\end{equation}
with $u:\dom\rightarrow\C\simeq\R^2$ and $\v$ is a positive parameter (the inverse of $\k$, the Ginzburg-Landau parameter). %Here we consider $\sqrt 2>\v>0$, then $E_\v$ models the state of a type II superconductor (see \cite{SS1}).%, the coherency length). %$\v>0$ is a parameter depending of the material. 

When there is no ambiguity we will write $E_\v(u)$ instead of $E_\v(u,\dom)$.

Functions we will consider belong to the class%In this article the set of test functions is
\[
\J=\left\{u\in H^1(\dom,\C)\,|\,|u|=1\text{ on }\p \dom\right\}.
\]
Clearly, $\J$ is closed under weak  $H^1-$convergence.

This functional is a simplified version of the Ginzburg-Landau functional which arises in superconductivity (or superfluidity) to model the state of a superconductor submitted to a magnetic field (see, \emph{e.g.}, \cite{Tin1} or \cite{SS1}). The simplified version of the Ginzburg-Landau functional considered in \eqref{3.1} ignores the magnetic field. The issue we consider in this article is existence of local minimizers with prescribed degrees on $\p\dom$.
%$\v$ behaves as the inverse of the temperature (see \cite{SS1}) and depends of the material. 
%The superconductivity phenomena appears in the superconductor only under a critical temperature (which varies with the material). Above this critical temperature, the superconductivity is destroyed. 

We next formulate rigorously the problem discussed in this article. To this purpose, we start by defining properly the degrees of a map $u\in\J$. For $\gamma\in\{\p\O,...,\p\omega_N\}$ and $u\in\J$ we let
\[
\deg_\gamma (u)=\frac{1}{2\pi}\int_\gamma{u\times\p_\tau u\di\tau}.%\in\Z.
\]
Here:
\begin{itemize}
\item each $\gamma$ is directly (counterclockwise) oriented,
\item $\tau=\nu^\bot$, $\tau$ is the tangential vector of $\gamma$ and $\nu$ the outward normal to $\O$ if $\gamma=\p\O$ or $\o_i$ if $\gamma=\p\o_i$, 
\item $\p_\tau=\tau\cdot\n$, the tangential derivative  and $"\cdot"$ stands for the scalar product in $\R^2$,
\item $"\times"$ stands for the vectorial product in $\C$, $(z_1+\imath z_2)\times(w_1+\imath w_2):=z_1w_2-z_2w_1,\:z_1,z_2,w_1,w_2\in\R$,  
\item the integral over $\gamma$ should be understood using the duality between $H^{1/2}(\gamma)$ and $H^{-1/2}(\gamma)$ (see, \emph{e.g.}, \cite{BeMi1} definition 1).
\end{itemize}

It is known that $\deg_\gamma (u)$ is an integer see \cite{BeMi1} (the introduction) or \cite{B1}.

We denote the (total) degree of $u\in\J$ in $\dom$ by
\[
\deg(u,\dom)=\left(\deg_{\p\omega_1}(u),...,\deg_{\p\omega_N}(u),\deg_{\p\Omega}(u)\right)\in\Z^N\times\Z.
\]

For $({\bf p},q)\in\Z^N\times\Z$, we are interested in the minimization of $E_\v$ in 
\[
\J_{{\bf p},q}:=\left\{u\in\J\,|\,\deg(u,\dom)=({\bf p},q)\right\}.
\] 

There is an huge literature devoted to the minimization of $E_\v$. In a simply connected domain $\O$, the minimization problem of $E_\v$ with the Dirichlet boundary condition $g\in C^\infty(\p\O, \mathbb{S}^1)$ is studied in detail in \cite{BBH}. $E_\v$ has a minimizer for each $\v>0$. This minimizer need not to be unique. In this framework, when $\deg_{\p\O}(g)\neq0$, the authors studied the asymptotic behaviour of a sequence of minimizers (when $\v_n\downarrow0$) and  point out the existence (up to subsequence) of a finite set of singularities of the limit.

Other types of boundary conditions were studied, like Dirichlet condition $g\in C^\infty(\p\O,\C\setminus\{0\})$ (in a simply connected domain $\O$) in \cite{AS1} and later for $g\in C^\infty(\p\O,\C)$ (see \cite{AS2}).

If the boundary data is not $u_{|\p\dom}$, but a given set of degrees, then the existence of local minimizers is non trivial. Indeed, one can show that $\J_{{\bf p},q}$ is not closed under weak $H^1$-convergence (see next section), so that one cannot apply the direct method in the calculus of variations in order to derive existence of minimizers. Actually this is not just a technical difficulty, since in general the {\it infimum} of $E_\v$ in $\J_{{\bf p},q}$ is {\bf not} attained, we need more assumptions like the value of the $H^1-$capacity of $\dom$ (see \cite{BGR1} and \cite{BeMi1}).

Minimizers $u$ of $E_\v$ in $\J_{{\bf p},q}$, if they do exist, satisfy the equation
\begin{equation}\label{3.Gleq}
\left\{
\begin{array}{cl}
\d-\Delta u=\frac{u}{\v^2}(1-|u|^2)&\text{in }\dom
\\
|u|=1&\text{on }\p\dom
\\
u\times\p_\nu u=0&\text{on }\p\dom
\\
\deg(u,\dom)=({\bf p},q)&
\end{array}
\right.
\end{equation}
where $\p_\nu$ denotes the normal derivative, \emph{i.e.}, $\d\p_\nu=\frac{\p}{\p \nu}=\nu\cdot\n$.

Existence of local minimizers of $E_\v$ is obtained following the same lines as in \cite{BeRy1}. It turns out that, even if the {\it infimum} of $E_\v$ in $\J_{{\bf p},q}$ is not attained, \eqref{3.Gleq} may have solutions. This was established by Berlyand and Rybalko when $\dom$ has a single hole, \emph{i.e.}, when $N=1$.
Our main result is the following generalisation of the main result in \cite{BeRy1}:
\begin{thm}\label{T3.main}
Let $({\bf p},q)\in\Z^N\times\Z$ and let $M\in\N^*$, there is $\v_{\etiquette{3.eps}{3.eps1}} ({\bf p},q,M)>0$  s.t. for $\v<\v_{\ref{3.eps1}}$, there are at least $M$ locally minimizing solutions.% of \eqref{3.Gleq}. 
\end{thm}
Actually, we will prove a more precise form of Theorem \ref{T3.main} (see Theorem \ref{T3.Existence}), whose statement relies on the notion of {\it approximate bulk degree} introduced in \cite{BeRy1} and generalised in the next section.

The main difference with respect to \cite{BeRy1} stems in the construction of the test functions with energy control in section \ref{3.sectionPREUVE}. In a sense that will be explained in details in section \ref{3.sectionPREUVE}, our construction is local, while the one in \cite{BeRy1} is global. We also simplify and unify some proofs in \cite{BeRy1}.

We do not know whether the conclusion of theorem \ref{T3.main} still holds when $\dom$ has no holes at all. That is, we do not know whether for a simply connected domain $\O$, a given $d\in\Z^*$ and small $\v$, the problem
\begin{equation}\label{3.Gleqbis}
\left\{
\begin{array}{cl}
\d-\Delta u=\frac{u}{\v^2}(1-|u|^2)&\text{in }\O
\\
u\times\p_\nu u=0&\text{on }\p\O
\\
|u|=1&\text{on }\p\O
\\

\deg_{\p\O}(u)=d&
\end{array}
\right.
\end{equation}
has solutions. Existence of a solution of \eqref{3.Gleqbis} is clear when $\O$ is a disc, say $\O=D(0,R)$ (it suffices to consider a solution of $-\Delta u=\frac{u}{\v^2}(1-|u|^2)$ of the form $\d u(z)=f(|z|)\left(\frac{z}{|z|}\right)^d$ with $\d u_{|\p\O}=\left(\frac{z}{|z|}\right)^d$). However, we do not know the answer when $\O$ is not radially symmetric anymore.

\section{The approximate bulk degree}
This section is a straightforward adaptation of \cite{BeRy1}.
 
Existence of (local) minimizers for $E_\v$ in $\J_{{\bf p},q}$ is not straightforward since $\J_{{\bf p},q}$ is not closed under weak $H^1-$convergence. A typical example (see \cite{BeMi1}) is a sequence $(M_n)_n$ s.t.
\[
\begin{array}{ccccc} M_n: &D(0,1)&\rightarrow &D(0,1)&\\&x&\mapsto&\displaystyle\frac{x-(1-1/n)}{(1-1/n)x-1}&\end{array},
\]
where $D(0,1)\subset\C$ is the open unit disc centered at the origin. Then $M_n\faible 1$ in $H^1$, $\deg_{\mathbb{S}^1}(M_n)=1$ and $\deg_{\mathbb{S}^1}(1)=0$.

To obtain local minimizers, Berlyand and Rybalko (in \cite{BeRy1}) devised a tool: the {\it approximate bulk degree}.    
We adapt this tool for a multiply connected domain.

We consider, for $i\in\N_N:=\{1,...,N\}$, $V_i$ the unique solution of 
\begin{equation}\label{3.3}
\left\{\begin{array}{cl}
-\Delta V_i=0&\text{in }\dom
\\
V_i=1&\text{on }\p \dom\setminus\p\o_i
\\
V_i=0&\text{on }\p \omega_i
\end{array}\right..
\end{equation}
For $u\in\J$, we set, noting $\d\p_k u=\frac{\p}{\p x_k}u$
\begin{equation}\label{3.4}
\ab_i(u,\dom)=\frac{1}{2\pi}\int_\dom{u\times(\p_1 V_i\,\p_2u-\p_2V_i\,\p_1u)\di x},
\end{equation}
\begin{equation}\no
\ab(u,\dom)=\,\left(\,\ab_1(u,\dom)\,,...,\, \ab_N(u,\dom)\,\right).
\end{equation}
Following \cite{BeRy1}, we call $\ab(u,\dom)$ the {\it approximate bulk degree} of $u$. $\ab_i:\J\rightarrow\R$, in general, is not an integer (unlike the degree). However, we have
\begin{prop}\label{P3.1}
\begin{enumerate}[1)]
\item If $u\in H^1(\dom,\mathbb{S}^1)$, then $\ab_i(u,\dom)=\deg_{\p\omega_i}(u)$;
\item Let $\Lambda,\v>0$ and $u,v\in\J$ s.t. $E_\v(u),E_\v(v)\leq\Lambda$, then
\begin{equation}\label{3.5}
|\ab_i(u)-\ab_i(v)|\leq\frac{2}{\pi}\|V_i\|_{C^1(\dom)}\Lambda^{1/2}\|u-v\|_{L^2(\dom)};
\end{equation}
\item Let $\Lambda>0$ and $(u_\v)_{\v>0}\subset\J$ s.t. for all $\v>0$, $E_\v(u_\v)\leq\Lambda$, then
\begin{equation}\label{3.6}
\dist(\ab(u_\v),\Z^N)\rightarrow0\text{ when }\v\rightarrow0.
\end{equation}
\end{enumerate}
\end{prop}
Proof of Proposition \ref{P3.1} is postponed to Appendix \ref{3.Apres}.

We define for $\mathbf{d}=(d_1,...,d_N)\in\Z^N$, $\bfp=(p_1,...,p_N)\in\Z^N$ and $q\in\Z$, 
\[
\d\J_{\bfp,q}^\db=\d\J_{\bfp,q}^\db(\dom):=\left\{u\in\J_{\bfp,q}\,|\,\|\ab(u)-\db\|_{\infty}:=\max_{i\in\N_N}{|d_i-\ab_i(u)|}\leq\frac{1}{3}\right\}.
\]
The following result states that $\Jp$ in never empty for $({\bf p},q,\db)\in\Z^N\times\Z\times\Z^N$.

\begin{prop}\label{P3.sens}
Let $({\bf p},q,\db)\in\Z^N\times\Z\times\Z^N$. Then $\Jp\neq\emptyset$.
\end{prop}
\begin{proof}For $i\in\{0,...,N\}$, we denote ${\bf e}_i=\left(\delta_{i,1},...,\delta_{i,N},\delta_{i,0}\right)\in\Z^{N+1}$ where 
\[
\delta_{i,k}=\left\{\begin{array}{cl}1&\text{if }i=k \\0&\text{otherwise}\end{array}\right.\text{ is the Kronecker symbol.}
\]
For $i\in\{0,...,N\}$, there is $M^i_n\in\J_{(p_i-d_i){\bf e}_i}$ if $i\neq0$ and $M^0_n\in\J_{(q-\sum d_j){\bf e}_0}$ s.t. $M^i_n\faible 1$ in $H^1$ and $|M^i_n|\leq1$ (Lemmas 6.1 and 6.2 in \cite{BeMi1}). Let
\[
E_\db:=\left\{u\in H^1(\dom,\mathbb{S}^1)\,|\,\deg(u,\dom)=(\db,d)\right\},\db=(d_1,...,d_N),\,d=\sum_{j=1}^Nd_j.
\]
We note that, $E_\db\neq\emptyset$, see, \emph{e.g.}, \cite{BBH}. Let $u\in E_\db$ and $u_n:=u\prod_{i=0}^NM_n^i$. Then we will prove that, for large $n$, we have, up to subsequence, that $u_n\in\Jp$.
Indeed, up to subsequence, 
\[
u_n\faible u\text{ in }H^1,\phantom{a}u_n\in\J_{{\bf p},q}.
\]
Using the fact that $\ab(u)=\db$ and the weak $H^1$-continuity of the {\it approximate bulk degree}, we obtain for $n$ sufficiently large, that $u_n\in\Jp$.
\end{proof}

We denote $m_\v(\bfp,q,\db)$ the \emph{infimum} of $E_\v$ on $\J_{\bfp,q}^\db$, \emph{i.e},
\[
m_\v({\bf p},q,\db)=\inf_{u\in\Jp}{E_\v(u)}
\]
and 
\[
I_0(\db,\dom)=\inf_{u\in E_{\db}}\frac{1}{2}\int_\dom|\n u|^2.
\]
We may now state a refined version of Theorem \ref{T3.main}.
\begin{thm}\label{T3.Existence}
Let ${\bf d }\in{(\N^*)}^{N}$. Then, for all $(p_1,...,p_N,q)\in\Z^{N+1}$
s.t. $q\leq d$ and $p_i\leq d_i$, there is $\v_{\etiquette{3.eps}{3.eps2}}=\v_{\ref{3.eps2}}({\bf p},q,\db)>0$ 
s.t. for $0<\v<\v_{\ref{3.eps2}}$, $m_\v(\bfp,q,\db)$ is attained.

Moreover, we have the following estimate 
\[
m_\v(\bfp,q,\db)= I_0(\db,\dom)+\pi\left(d_1-p_1+...+d_N-p_N+d-q\right)-o_\v(1),\:\:o_\v(1)\underset{\v\rightarrow0}{\rightarrow}0.
\]
\end{thm}
For further use, a configuration of degrees $({\bf p},q,\db)\in\Z^N\times\Z\times(\N^*)^N$ s.t. $p_i\leq d_i$ and $q\leq \sum d_i$ will be called a "good configuration".
Noting that, for $\db\neq\tilde{\db}\in\Z^N$ and $({\bf p},q)\in\Z^N\times\Z$, we have $\Jp\cap\J_{{\bf p},q}^{\tilde{\db}}=\emptyset$, we are led to
\\ {\bf Proof of Theorem \ref{T3.main}:} Let $({\bf p},q)\in\Z^N\times\Z$ and set $\text{for }k\in\N^*$,
\[
\d d=\max\left\{\max_i|p_i|,|q|\right\}\text{ and }\db_k=(d+k,...,d+k). 
\]
We apply Theorem \ref{T3.Existence} to the class $\J_{{\bf p},q}^{\db_k}$. We obtain the existence of 
\[
\v_{\ref{3.eps1}}({\bf p},q,M)=\min_{k\in\N_M}{\v_{\ref{3.eps2}}({\bf p},q,\db_k)}>0
\]
s.t. for $\v<\v_{\ref{3.eps1}}$, $k\in\N_M$, $m_\v({\bf p},q,\db_k)$ is achieved by $u^k_\v$.

Noting the continuity of the degree and of the {\it approximate bulk degree} for the strong $H^1$-convergence, there exists $V^k_\v\subset \J_{{\bf p},q}^{\db_k}\subset \J$ an open (for $H^1$-norm) neighbourhood of $u_\v^k$. It follows easily that
\[
E_\v(u^k_\v)=\min_{u\in V^k_\v}{E_\v(u)}.
\] 
Then $u^k_\v\in\J_{{\bf p},q}$ is a local minimizer of $E_\v$ in $\J$ (for $H^1$-norm) for $0<\v<\v_{\ref{3.eps1}}({\bf p},q,M)$.

\section{Basic facts of the Ginzburg-Landau theory}\label{3.BASIC}
It is well known (\emph{cf} \cite{BeMi1}, lemma 4.4 page 22) that the local minimizers of $E_\v$ in $\J_{\bfp,q}$ satisfy 
\begin{equation}\label{3.8}
-\Delta u=\frac{1}{\v^2}u(1-|u|^2)\text{ in }\dom,
\end{equation}
\begin{equation}\label{3.8'}
|u|=1\text{ and }u\times\p_\nu u=0\text{ on }\p \dom.
\end{equation}

Equation (\ref{3.8}) and the Dirichlet condition on the modulus in (\ref{3.8'}) are classical. The Neumann condition on the phase in \eqref{3.8'} is less standard but it is for example stated in \cite{BeMi1}.

Equation (\ref{3.8}) combined with the boundary condition on $\p \dom$ implies, \emph{via} a maximum principle, that 
\begin{equation}\label{3.9}
|u|\leq1\text{ in }\dom.
\end{equation}

One of the questions in the Ginzburg-Landau model is the location of the vortices of stable solutions (\emph{i.e.}, local minimizers of $E_\v$).  We will define \emph{ad hoc }a vortex as an isolated zero $x$ of $u$ with nonzero degree on small circles around $x$.
 
 The following result shows that, under energy bound assumptions on solutions of (\ref{3.8}), vortices are expelled to the boundary when $\v\to0$.
\begin{lem}\label{L3.2}\cite{M1}
Let $\Lambda>0$ and let $u$ be a solution of (\ref{3.8}) satisfying (\ref{3.9}) and the energy bound $E_{\v}(u)\leq\Lambda$. Then with $C,C_k$ and $\v_{\etiquette{3.eps}{3.epsMi}}$ depending only on $\Lambda$, $\dom$, we have, for $0<\v<\v_{\ref{3.epsMi}}$ and $x\in\dom$,
\begin{equation}\label{3.10}
1-|u(x)|^2\leq\frac{C\v^2}{\dist^2(x,\p \dom)}
\end{equation}
and
\begin{equation}\label{3.11}
|D^ku(x)|\leq\frac{C_k}{\dist^k(x,\p \dom)}.
\end{equation}
\end{lem} 

When $u$ is smooth in $\dom$ and $\rho=|u|>0$, the map $\d \frac{u}{\rho}$ admits a lifting $\theta$ , \emph{i.e}, we may write
\[
u=\rho\e^{\imath\theta},
\]where $\theta$ is a smooth (and locally defined) real function on $\dom$ and $\n\theta$ is a globally defined smooth vector field.

Using (\ref{3.8}) and (\ref{3.8'}), we have 
\begin{equation}\label{3.40}
\left\{\begin{array}{cl}
\Div(\rho^2\n\theta)=0&\text{in }B
\\
\p_\nu\theta=0&\text{on }\p \dom
\end{array}\right.,
\end{equation}
\begin{equation}\label{3.41}
\left\{\begin{array}{cl}
\d-\Delta\rho+|\n\theta|^2\rho+\frac{1}{\v^2}\rho(\rho^2-1)=0&\text{in }B
\\
\rho=1&\text{on }\p \dom
\end{array}\right.,
\end{equation}
here, $B=\{x\in\dom\,|\,u(x)\neq0\}$.

We will need later the following. 
\begin{lem}\label{L3.16} \cite{BeRy1}
Let $u$ be a solution of (\ref{3.8}) and (\ref{3.8'}). Let $G\subset \dom$ be an open Lipschitz set s.t. $u$ does not vanish in $\overline{G}$. Write, in $\overline{G}$, $u=\rho v$ with $\rho=|u|$. Let $w\in H^1(G,\C)$ be s.t. $|\tr_{\p G}w|\equiv1$. Then 
\[
E_\v(\rho w,G)=E_\v(u,G)+L_\v(w,G),
\]
with
\[
L_\v(w,G)=\frac{1}{2}\int_{G}{\rho^2|\n w|^2\di x}-\frac{1}{2}\int_{G}{|w|^2\rho^2|\n v|^2\di x}+\frac{1}{4\v^2}\int_{G}{\rho^4(1-|w|^2)^2\di x}.
\]
\end{lem}
For further use, we note that we may write, locally in $\overline{G}$, $u=\rho\e^{\imath\theta}$, so that $v=\e^{\imath\theta}$. It turns out that $\n\theta$ is smooth and globally defined in $\overline{G}$. In terms of $\n\theta$, we may rewrite
\[
L_\v(w,G)=\frac{1}{2}\int_{G}{\rho^2|\n w|^2\di x}-\frac{1}{2}\int_{G}{|w|^2\rho^2|\n \theta|^2\di x}+\frac{1}{4\v^2}\int_{G}{\rho^4(1-|w|^2)^2\di x}.
\]
For $u$ a solution of  (\ref{3.8}) and (\ref{3.8'}), we can consider (see Lemma 7 in \cite{BeRy1}) $h$ the unique globally defined solution of 
\begin{equation}\label{3.42}
\left\{\begin{array}{cl}\n^\bot h=u\times\n u&\text{in }\dom\\h=1&\text{on }\p \O\\h=k_i&\text{on }\p\omega_i\end{array}\right.,
\end{equation}
where $k_i$'s are real constants uniquely defined by the first two equations in (\ref{3.42}). Here
\[
\n^\bot h=\left(\begin{array}{c}-\p_2h\\\p_1h\end{array}\right)\text{ is the orthogonal gradient of }h\text{ and }u\times\n u=\left(\begin{array}{c}u\times\p_1 u\\u\times\p_2 u\end{array}\right).
\]

It is easy to show that
\begin{equation}\label{3.43}
\left\{\begin{array}{cl}\n h=-\rho^2\n^\bot\theta&\text{in }B\\\d\Div(\frac{1}{\rho^2}\n h)=0&\text{in }B\\\Delta h=2\p_1u\,\times\,\p_2u&\text{in }B\end{array}\right.;
\end{equation}
here, $B=\{x\in\dom\,|\,u(x)\neq0\}$.

In \cite{BBH}, Bethuel, Brezis and Hélein consider the minimization of $\d E(u)=\frac{1}{2}\int_\dom{|\n u|^2\di x}$, the Dirichlet functional, in the class 
\[
 E_{\db }=\{u\in H^1(\dom,\mathbb{S}^1)\,|\,\deg(u,\dom)=(\db,d)\};
\]
here, $d=\sum d_k$.% \footnote{Un résultat classique affirme que $E_{d_1,d_2,d}\neq\emptyset\Leftrightarrow d_1+d_2=d$.}

Theorem I.1 in \cite{BBH} gives the existence of a unique solution (up to multiplication by an $\mathbb{S}^1$-constant) for the minimization of $E$ in $E_{\db}$. We denote $u_0$ this solution. This $u_0$ is also a solution of
\[
\left\{\begin{array}{cl}
  -\Delta v  =v|\n v|^2 &\text{in } \dom   \\
     \d v\times\p_\nu v=0&\text{on }\p \dom   \\
\end{array}\right..
\]
Moreover, we have
\begin{equation}\label{3.17}
I_0(\db ,\dom):=\min_{u\in E_{\db }}{E(u)}=\frac{1}{2}\int_\dom{|\n h_0|^2\di x}
\end{equation}
with $h_0$ the unique solution of
\begin{equation}\label{3.18}
\left\{\begin{array}{cl}\Delta h_0=0&\text{in }\dom\\h_0=1&\text{on }\p\O\\h_0=\text{Cst}_k&\text{on }\p\omega_k,\,k\in\{1,...,N\}\\\d\int_{\p\omega_k}{\p_\nu h_0\di \sigma}=2\pi d_{k}&\text{for }k\in\{1,...,N\}\end{array}\right..
\end{equation}
One may prove that $h_0$ is the (globally defined) harmonic conjugate of a local lifting of $u_0$.

\section{Energy needed to change degrees}
We denote %suivant le contexte
%\[
%\begin{array}{cccc}\tae:&\Z^N\times\Z^N&\rightarrow&\N
%\\
%&(\db,{\bf p})&\mapsto&\textsf{\ae}(\db,{\bf p})=\sum_{i=1}^N{|d_i-p_i|}
%\end{array}
%\]
%ou
\[
\begin{array}{cccc}\tae:&(\Z^N\times\Z)\times(\Z^N\times\Z)&\rightarrow&\N
\\
&\left((\db,d),({\bf p},q)\right)&\mapsto&\sum_{i=1}^N{|d_i-p_i|}+|d-q|
\end{array}.
\]
The next result quantifies the energy needed to change degrees in the weak limit.%expresses the minimal loss of energy during a weak convergence changing of degrees class. 
\begin{lem}(\label{L3.6}\cite{BeMi1}, Lemma 1)
Let $(u_n)_n\subset\J_{\bfp,q}$ be a sequence weakly converging in $H^1$ to $u$. Then
\begin{equation}\label{3.33}
\liminf_n{E(u_n)}\geq E(u)+\pi\tae(\deg(u,\dom),({\bf p},q))
\end{equation}
and for $\v>0$%Using Sobolev inequalities we have from (\ref{3.33}) 
\begin{equation}\label{3.34}
\liminf_n{E_\v(u_n)}\geq E_\v(u)+\pi\tae(\deg(u,\dom),({\bf p},q)).
\end{equation}
\end{lem}
The next lemma is proved in \cite{BeRy1}.
\begin{lem}\label{L3.5}
Let $\db=(d_1,...,d_N),{\bf p}=(p_1,...,p_N)\in\Z^N,\,q\in\Z$. There is $o_\v(1)\underset{\v\to0}{\to}0$ (depending of $({\bf p},q,\db)$) s.t. for $u\in  \J_{{\bf p},q}^\db$ we have
\begin{equation}\label{3.19}
E_\v(u)\geq I_0(\db ,\dom)+\pi\textsf{\ae}((\db,d),({\bf p},q))-o_\v(1).
\end{equation}
Here, $d:=\sum d_i$.
\end{lem}
We present below a simpler proof than the original one in \cite{BeRy1}. 

\begin{proof}
Let $({\bf p},q,{\bf d})\in\Z^N\times\Z\times\Z^N$. We argue by contradiction and we suppose that there are $\delta>0$, $\v_n\downarrow0$ and $(u_n)_n\subset\mathcal{J}_{{\bf p},q}^\db$ s.t.
\begin{equation}\label{3.bornepreuve}
E_{\v_n}(u_n)\leq I_0(\db ,\dom)+\pi\textsf{\ae}((\db,d),({\bf p},q))-\delta.
\end{equation}
Since $(u_n)_n$ is bounded in $H^1$, there is some $u$ s.t., up to subsequence, $u_n\faible u$ in $H^1$ and $u_n\rightarrow u$ in $L^4$. Using the strong convergence in $L^4$, (\ref{3.bornepreuve}) and Proposition \ref{P3.1}, we have $u\in H^1(\dom,\mathbb{S}^1)\cap\mathcal{J}_{\db,d}^\db=E_{\bf d}$.

To conclude, we use \eqref{3.bornepreuve} combined with Lemma \ref{L3.6}
\begin{eqnarray*}
I_0(\db ,\dom)+\pi\textsf{\ae}((\db,d),({\bf p},q))-\delta&\geq& \liminf_n E_{\v_n}(u_n)
\\
&\geq&\liminf_n E(u_n)
\\
&\geq&E(u)+\pi\tae((\db,d),({\bf p},q))
\\
&\geq&I_0(\db,\dom)+\pi\tae((\db,d),({\bf p},q))%\text{ this is a contradiction}.
\end{eqnarray*}
which is a contradiction.
\end{proof}
%For $i\in\{0,...,N\}$, we denote ${\bf e}_i=\left(\left(\delta_{i,1},...,\delta_{i,N}\right),\delta_{i,0}\right)$ where $\delta_{i,k}=\left\{\begin{array}{cl}1&\text{if }i=k, \\0&\text{ otherwise}\end{array}\right.$. It is the Kronecker symbol.
One may easily proved (see Lemma \ref{L3.bougedegfinal} in Appendix \ref{3.bougerledeg}) that for $\eta>0$, $i\in\{0,...,N\}$ and $u\in\J_{\deg(u,\dom)}$, there are $v_\pm\in\J_{\deg(u,\dom)\pm{\bf e}_i}$ s.t.
\[
E_\v(v_\pm)\leq E_\v(u)+\pi+\eta.
\]
The key ingredient is a sharper result which holds under two additional hypotheses. In order to unify the notations, we use the notation $\omega_0$ for $\O$. We may now state the main ingredient in the proof of Theorem \ref{T3.Existence}.
\begin{lem}\label{L3.15soft}
Let $u\in \J_{{\bf p},q}$ be a solution of (\ref{3.8}), (\ref{3.8'}). 

Assume that 
\begin{equation}\label{3.deg}
\ab_j(u)\in(d_j-\frac{1}{3},d_j+\frac{1}{3}),\:\forall\,j\in\N_N.
\end{equation}
Let $i\in\{0,...,N\}$ and assume that there is some point $x^i\in\p\omega_i$ s.t. $ u\times\p_\tau u(x^i)>0$. Recall that $\tau$ is the direct tangent vector to $\p\o_i$.

Then there is $\tilde{u}\in\J_{({\bf p},q)-{\bf e}_i}$ s.t. 
\begin{equation}\no
E_\v(\tilde{u})<E_\v(u)+\pi
\end{equation}
and
\[
\ab_j(\tilde{u})\in(d_j-\frac{1}{3},d_j+\frac{1}{3}),\,\forall j\in\N_N.
\]
\end{lem}
The proof of Lemma \ref{L3.15soft} is postponed to section \ref{3.sectionPREUVE}.

We also have an upper bound for $m_\v({\bf p},q,\db)$.
\begin{lem}\label{L3.bornesup}
Let $\v>0$ and $({\bf p},q,\db)\in\Z^N\times\Z\times\Z^N$. Then
\begin{equation}\label{3.supmv}
m_\v({\bf p},q,\db)\leq I_0(\db,\dom)+\pi\tae((\db,d),({\bf p},q)).
\end{equation}
\end{lem}
To prove Lemma \ref{L3.bornesup}, we need the following
\begin{lem}\label{P3.bougedegmult}
Let $u\in\J$, $\v>0$ and ${\bf \delta}=(\delta_1,...,\delta_N,\delta_0)\in\Z^{N+1}$. For all $\eta>0$, there is $u_\eta^\delta\in\J_{\deg(u,\dom)+\delta}$ s.t.
\begin{equation}\label{3.bougedegetamul}
E_\v(u_\eta^\delta)\leq E_\v(u)+\pi\sum_{i\in\{0,...,N\}}{|\delta_i|}+\eta
\end{equation}

and
\begin{equation}\label{3.convennormeL2mul}
\|u-u_\eta^\delta\|_{L^2(\dom)}=o_\eta(1),\:o_\eta(1)\underset{\eta\goto0}{\goto}0.
\end{equation}
\end{lem}
The proof of Lemma \ref{P3.bougedegmult} is postponed to Appendix \ref{3.bougerledeg}.
\begin{proof}
We prove that for $\eta>0$ small, we have
\[
m_\v({\bf p},q,\db)\leq I_0(\db,\dom)+\pi\tae((\db,d),({\bf p},q))+\eta.
\]
We denote $u_0\in E_\db$ s.t. $E(u_0)=I_0(\db,\dom)$. Then $\ab_i(u_0)=d_i$. 

Using Lemma \ref{P3.bougedegmult} with $\delta=({\bf p},q)-({\bf d},d)$, there is $u_\eta$ s.t. 
\[
u_\eta\in\J_{({\bf p},q)}\text{ and }E_\v(u_\eta)\leq E_\v(u_0)+\pi\tae((\db,d),({\bf p},q))+\eta=I_0(\db,\dom)+\pi\tae((\db,d),({\bf p},q))+\eta.
\] 
Furthermore, by \eqref{3.convennormeL2mul}, $\|u_0-u_\eta\|_{L^2(\dom)}=o_\eta(1)$. For $\eta$ small, by Proposition \ref{P3.1}, we have $u_0\in\Jp$ which proves the lemma.
\end{proof}

\section{A family  with bounded energy converges}

In this section we discuss:
\begin{enumerate}  \item the asymptotic behaviour of a sequence of solutions of (\ref{3.8}), (\ref{3.8'}), $(u_{\v_n})_n\subset\Jp$ ($\v_n\downarrow0$) with bounded energy , \emph{i.e}, $E_{\v_n}(u_{\v_n})\leq\Lambda$, 
\item the asymptotic behaviour of a minimizing sequence of $E_\v$ in $\Jp$,
\item a fundamental lemma.

\end{enumerate}
\begin{prop}\label{P3.conv}
Let $\v_n\downarrow0$, $(u_{\v_n})_n\subset\Jp$  with $u_{\v_n}$ a solution of (\ref{3.8}), (\ref{3.8'}),  s.t. for $\Lambda>0$, we have
\[
E_{\v_n}(u_{\v_n})\leq\Lambda.
\]
Then, denoting $h_{\v_n}$ the unique solution of (\ref{3.42}) with $u=u_{\v_n}$, we have
\begin{equation}\label{3.convh}
h_{\v_n}\faible h_0\text{ in }H^1(\dom),
\end{equation}
where $h_0$ is the unique solution of (\ref{3.18}).
\\Up to subsequence, it holds
\begin{equation}\label{3.convu}
u_{\v_n}\faible u_0\text{ in }H^1(\dom),
\end{equation}
where $u_0 \in E_{\db} $ is the unique solution of (\ref{3.17}) up to multiplication by an $\mathbb{S}^1$-constant.
\end{prop}
\begin{proof}
Using the energy bound on $u_{\v_n}$ and a Poincaré type inequality, we have, up to subsequence, 
\[
h_{\v_n}\faible h \text{ in }H^1.  
\]

In order to establish (\ref{3.convh}), it suffices to prove that $h=h_0$. 

The set $\mathcal{H}:=\{h\in H^1(\dom,\R)\,;\,\p_\tau h\equiv0\text{ on }\p \dom\text{ and }h_{|\p\O}\equiv1\}$ is closed convex in $H^1(\dom,\R)$. Since $(h_{\v_n})_n\subset\mathcal{H}$, we find that $h\in\mathcal{H}$.

By boundedness  of $E_{\v_n}(u_{\v_n})$, Lemma \ref{L3.2} implies that $u_{\v_n}$ is bounded in $C_{\rm loc}^2(\dom,\R^2)$. Therefore there is some $u\in C^1_{{\rm loc}}(\dom,\C)$ s.t., up to subsequence, $u_{\v_n}\rightarrow u$ in $C_{\rm loc}^1(\dom,\R^2)$, $L^4(\dom,\R^2)$ and weakly in $H^1(\dom,\R^2)$.

Using the strong convergence in $L^4$ and the energy bound on $u_{\v_n}$, we find that $u\in H^1(\dom,\mathbb{S}^1)$. It follows that $\p_1u\times\p_2 u=0$ in $\dom$. On the other hand, 
\[
\Delta h_{\v_n}=2\p_1u_{\v_n}\times\p_2 u_{\v_n}\rightarrow0\text{ in } C_{\rm loc}^0.
\]
Therefore, $h$ is a harmonic function in $\dom$.

In order to show that $h=h_0$, it suffices to check that
\[
\d\int_{\p\omega_i}{\p_\nu h\di\sigma}=2\pi d_i.
\]
To this end, we note that, since $u_{\v_n}\times(\p_1 V_i\p_2 u_{\v_n}-\p_2 V_i\p_1 u_{\v_n})=\n V_i\cdot\n h_{\v_n}$, we have from \eqref{3.3}
\[
2\pi\,\ab_i(u_{\v_n})=\int_\dom{\n V_i\cdot\n h_{\v_n}\di x}\xrightarrow[n\rightarrow\infty]{}\int_\dom{\n V_i\cdot\n h\di x}=\int_{\p \dom\setminus\p \omega_i}{\p_\nu h\di \sigma}.
\]
Noting that, by Proposition \ref{P3.1}, 
\[
\left\{\begin{array}{l}
\ab_i(u_{\v_n})\xrightarrow[n\rightarrow\infty]{}\ab_i(u)=\deg_{\p\omega_i}(u)
\\
\ab_i(u_{\v_n})\xrightarrow[n\rightarrow\infty]{} d_i
\end{array}\right.
\] 
and that $\d 0=\int_\dom{\Delta h\di x}=\int_{\p \dom}{\p_\nu h\di \sigma}$, we obtain
\[
\int_{\p \dom\setminus\p \omega_i}{\p_\nu h\di \sigma}=\int_{\p \omega_i}{\p_\nu h\di \sigma}=2\pi \,d_i=2\pi\,\deg_{\p\omega_i}(u).
\]
In the first integral, $\nu$ is the outward normal to $\dom$, in the second, $\nu$ is the outward normal to $\omega_i$.

This proves (\ref{3.convh}). 

We next turn to (\ref{3.convu}). Let $u_0$ be s.t., up to subsequence, $u_{\v_n}\faible u_0$ in $H^1(\dom)$. Since $|u_{\v_n}|\leq1$, we find that
\[
u_{\v_n}\times\n u_{\v_n}\faible u_0\times\n u_0\text{ in }L^2(\dom).
\]
In view of \eqref{3.42} and \eqref{3.convh}, we have $u_0\times\n u_0=\n^\bot h_0$.
Therefore,
\[
E(u_0)=E(h_0)=I_0(\db,\dom).
\]
Proposition \ref{P3.1} implies that $u_0\in E_\db$. Then $u_0$ is the unique, up to multiplication by an $\mathbb{S}^1$-constant, minimizer of $E$ in $E_\db$.
\end{proof}
\begin{prop}\label{P3.min} 
Let $({\bf p},q,\db)\in\Z^N\times\Z\times\Z^N$. For $\v>0$, let $(u^\v_n)_{n\geq0}\subset\Jp$ be a minimizing sequence of $E_\v$ in $\Jp$. Then there is $\v_{\etiquette{3.eps}{3.eps4}} \left({\bf p},q,\db\right)>0$ s.t. for $0<\v<\v_{\ref{3.eps4}}$, up to subsequence, $u_n\faible u$ in $H^1$ with $u$ which minimizes $E_\v$ in $\J_{\deg(u,\dom)}^\db$.
\end{prop}
\begin{proof}For $\v>0$, let $(u_n^{\v})_n\subset\Jp$ be a minimizing sequence of $E_\v$ in $\J$. Up to subsequence, using Proposition \ref{P3.1},
\[
u_n^{\v}\faible u^\v \text{ in }H^1\text{ with }u^\v\in\J^\db_{\deg(u^\v,\dom)}.
\]

Using Lemmas \ref{L3.6} and \ref{L3.bornesup}, we see that $\{\deg(u^\v,\dom),\v>0\}\subset\Z^N\times\Z$ is a finite set and that $E_\v(u^\v)$ is bounded. Therefore, with Proposition \ref{P3.1}, there is $\v_{\ref{3.eps4}}>0$ s.t. $|\ab_i(u^\v)-d_i|<\frac{1}{3}$ for all $i\in\N_N$ and $0<\v<\v_{\ref{3.eps4}}$.

We argue by contradiction and we assume that there is $\v<\v_{\ref{3.eps4}}$ s.t. 
\[
E_\v(u^\v)=m_\v(\deg(u^\v,\dom),\db)+2\eta,\,\eta>0.
\]
Let $u\in\J_{\deg(u^\v,\dom)}^\db$ be s.t. $E_\v(u)\leq m_\v(\deg(u^\v,\dom),\db)+\eta$. 

Using Lemma \ref{P3.bougedegmult} with $\delta=({\bf p},q)-\deg(u^\v,\dom)$, there is $v\in\J_{{\bf p},q}$ s.t. 
\[
E_\v(v)< E_\v(u)+\pi\tae(({\bf p},q),\deg(u^\v,\dom))+\eta.
\]Furthermore, by \eqref{3.convennormeL2mul}, $\|u-v\|_{L^2}$ can be taken arbitrary small, so that we may further assume $v\in\Jp$. To summarise we have
\begin{eqnarray*}
m_\v({\bf p},q,\db)&=&\liminf_nE_\v(u^\v_n)\\&\geq& E_\v(u^\v)+\pi\tae(({\bf p},q),\deg(u^\v,\dom))
\\&=& m_\v(\deg(u^\v,\dom),\db)+2\eta+\pi\tae(({\bf p},q),\deg(u^\v,\dom))
\\&\geq&E_\v(u)+\pi\tae(({\bf p},q),\deg(u^\v,\dom))+\eta
\\&>&E_\v(v)\geq m_\v({\bf p},q,\db).
\end{eqnarray*}
This contradiction completes the proof.

\end{proof}

The main tool requires the following lemma. 
\begin{lem}\label{L3.deg}
Let $({\bf p},q,\db)\in\Z^N\times\Z\times\Z^N$ and $\Lambda>0$. There is $\v_{\etiquette{3.eps}{3.eps5}}({\bf p},q,\db,\Lambda)>0$ s.t. for $\v<\v_{\ref{3.eps5}}$ and $u\in\Jp$, a solution of (\ref{3.8}) and (\ref{3.8'}) with $E_\v(u)\leq\Lambda$, if $d>0$ (respectively $d_i>0$), then there is $x^0\in\p\O$ (respectively $x^i\in\p\omega_i$) s.t. $u\times\p_\tau u(x^0)>0$ (respectively $u\times\p_\tau u(x^i)>0$).

Here $\tau$ is the direct tangent vector to $\p\O$ (resp. $\p\o_i$).%, $l\in\{0,1,...,N\}$ %(respectively $\p_\nu h_\v(x_\v^l)<0$)
 %ith, in the last case, $\nu$ which is the normal outward of $\omega_i$.
\end{lem}
\begin{proof}
We prove existence of $x^0\in\p\O$ under appropriate assumptions. Existence of $x^i$ is similar.

We argue by contradiction. Assume that there are $\v_n\downarrow0$, $(u_n)\subset\Jp$ solutions of (\ref{3.8}) and (\ref{3.8'}) with $E_{\v_n}(u_n)\leq\Lambda$ s.t. $ u_n\times\p_\tau u_n\leq0$ on $\p\O$. 

Since $\d q=\frac{1}{2\pi}\int_{\p\O}u_n\times\p_\tau u_n$, we have $q\leq0$.

Up to subsequence, by Proposition \ref{P3.conv}, we can assume that
\[
u_n\rightarrow u_0\text{ a.e.}\text{ with} \: u_0\text{ the unique solution (up to $\mathbb{S}^1$) of }(\ref{3.17}).
\]

Let $x_0\in\p\O$ and let $\gamma:\p\O\rightarrow[0,\H^1(\p\O)[=:I$ be s.t. $\gamma^{-1}$ is the direct arc-length parametrization of $\p\O$ with the origin at $x_0$. 

We denote $\theta_n:I\rightarrow\R$ the smooth functions s.t.
\[
\left\{
\begin{array}{c}
u_n(x)=\e^{\imath\theta_n[\gamma(x)]}\:\forall\,x\in\p\O
\\
0\leq\theta_n(0)<2\pi
\end{array}  .
\right.
\]
Then, for all $n$, $\theta_n$ is nonincreasing and $\theta_n\in[\theta_n(0)+2\pi q,\theta_n(0)]\subset[2\pi q,2\pi]$.% for all $n$. 

Using Helly's selection theorem, up to subsequence, we can assume that $\theta_n\rightarrow\theta$ everywhere on $I$ with $\theta$ nonincreasing. Denote $\Xi$ the set of discontinuity points of $\theta$. Since $\theta$ is nonincreasing, $\Xi$ is a countable set.

Using the monotonicity of $\theta$, we can consider the following decomposition
\[
\theta=\theta^c+\theta^\delta,\text{ with }\theta^c\text{ and }\theta^\delta\text{ are nonincreasing functions}.
\]
$\theta^c$ is the continuous part of $\theta$ and $\theta^\delta $ is the jump function. The set of discontinuity points of $\theta^\delta$ is $\Xi$.

For $t\notin\Xi$, 
\[
\theta^\delta(t)=%\left(\theta(t)-\theta(t^-)\right)\delta_t+
\sum_{0<s<t,\,s\in\,\Xi}{\left\{\theta(s+)-\theta(s-)\right\}}.
\]

We obtain easily that $u_0(x)=\e^{\imath\theta[\gamma(x)]}$ a.e. $x\in\p\O$. Since $u_0$, $\theta_n$ and $\gamma$ have side limits at each points and $u_0=\e^{\imath\theta\circ\gamma}$ a.e., we find that
\[
u_0(x\pm)=\e^{\imath\theta[\gamma(x\pm)]}\text{ for {\bf each} }x\in\p\O.
\]
Using the continuity of $u_0$, we obtain $\e^{\imath\theta[\gamma(x+)]}=\e^{\imath\theta[\gamma(x-)]}\,\forall\,x\in\p\O$ which implies that 
\[
\theta[\gamma(x+)]-\theta[\gamma(x-)]\in2\pi\Z\:\forall\,x\in\p\O.
\]
For $t\notin\Xi$, 
\[
\theta^\delta(t)=\sum_{0<s<t,\,s\in\,\Xi}{\left\{\theta(s+)-\theta(s-)\right\}}\in2\pi\Z.
\]
 Then
\[
 u_0(x)\e^{-\imath\theta^c[\gamma(x)]}= \e^{\imath\theta^\delta[\gamma(x)]}=1\text{ a.e. }x\in\p\O.
\]
Finally, $u_0(x)=\e^{\imath\theta^c[\gamma(x)]}$ a.e. $x\in\p\O$, which is equivalent (using the continuity of the functions) at $u_0=\e^{\imath\theta^c\circ\gamma}$.

We have a contradiction observing that 
\[
0<2\pi\deg_{\p\O}(u_0)=2\pi d=\theta^c(\H^1(\p\O))-\theta^c(0)
\]
and using the fact that $\theta^c$ is nonincreasing.
\end{proof}

\section{Proof of Lemma \ref{L3.15soft}}\label{3.sectionPREUVE}

We prove only the part of the lemma concerning $\p\O$. The proof for the other connected components of $\p \dom$ is similar. 

For reader's convenience, we state the part of Lemma \ref{L3.15soft} that we will actually prove

\begin{lemimnoi}
Let $u\in \J_{{\bf p},q}$ be a solution of (\ref{3.8}) and (\ref{3.8'}). 

Assume that 
\begin{equation}\tag{\ref{3.deg}}
\ab_j(u)\in(d_j-\frac{1}{3},d_j+\frac{1}{3}),\:\:\forall\,j\in\N_N
\end{equation}
and that there is some point $x^0\in\p\O$ s.t. $u\times\p_\tau u(x^0)>0$. 

Then there is $\tilde{u}\in\J_{({\bf p},q-1)}$ s.t. 
\begin{equation}\no
E_\v(\tilde{u})<E_\v(u)+\pi,
\end{equation}
\[
\ab_j(\tilde{u})\in(d_j-\frac{1}{3},d_j+\frac{1}{3}),\,\forall\, j\in\N_N.
\]

\end{lemimnoi}

\subsection{Decomposition of $\dom$}

By hypothesis, there is some $x^0\in\p\O$ s.t.  $\p_\nu h(x^0)>0$. Without loss of generality, we may assume that $u(x^0)=1$.

Then there is $\Upsilon\subset \overline{\dom}$, a compact neighbourhood of $x^0$, simply connected and with nonempty interior,  s.t.:
\begin{itemize}
\item $\gamma:=\p\O\cap\p\Upsilon$ is connected with nonempty interior;
\item $x^0$ is an interior point of $\gamma$;
\item $|\n h|>0$, $\rho>0$, $h\leq1$ in $\Upsilon$;
\item $\p_\nu h > 0$ on $\gamma$ ($\nu$ the outward normal to $\O$).
\end{itemize}

It follows that, in $\Upsilon$, $\theta$, a lifting of $u/|u|$ is globally defined (we take the determination of $\theta$ which vanishes at $x^0$) .

Using the inverse function theorem, we may assume, by further restricting $\Upsilon$, that there are some $0<\eta,\delta<1$ s.t. 
\[
\Upsilon=\{x\in \dom\text{ s.t. } \dist(x,x^0)<\eta,\,1-\delta\leq h(x)\leq1,\,-2\delta\leq\theta(x)\leq2\delta\}.
\]

We may further assume that, by replacing $\delta$ by smaller value if necessary and denoting $D_\delta:=\stackrel{\circ}{\Upsilon}$ (see Figure \ref{F3.Fig.LocalMin}), we have
\begin{enumerate}[(i)]\item $\begin{array}{ccccc}\Theta:=(\theta,h)_{|D_\delta}:&D_\delta&\rightarrow&(-2\delta,2\delta)\times(1-\delta,1)&\text{ is a }C^1\text{-diffeomorphism},\\&x&\mapsto&(\theta,h)&\end{array}$
\item $\p D_\delta\setminus(\{h=1\}\cup\{h=1-\delta\})=\p D_\delta\cap(\{\theta=-2\delta\}\cup\{\theta=2\delta\})$,
\item $D_\delta$  is a Lipschitz domain.%pas forcément régulier vu que sa frontière est \emph{a priori} continue et régulière par morceaux.
\end{enumerate}
We consider $\delta_0>0$ s.t. for $\delta<\delta_0$, $D_\delta$ satisfies previous properties and
\begin{equation}\label{3.83}
|D_\delta|^{1/2}<\frac{\pi\left|\|\ab(u)-{\bf d}\|_{\infty}-\frac{1}{3}\right|}{6\max_i\|V_i\|_{C^1(\dom)}\left(E_\v(u)+\pi\right)^{1/2}}.
\end{equation}
Using Proposition \ref{P3.1} and (\ref{3.83}), if $v\in H^1(\dom,\C)$ satisfies $u=v$ in $\dom\setminus D_\delta$, $|v|\leq2$ in $\dom$ and $E_\v(v)<E_\v(u)+\pi$, then we have $\ab_i(v)\in(d_i-1/3,d_i+1/3)$. %{\tt Pour obtenir ce résultat une borne d'énergie aussi fine que celle du lemme \ref{L3.15} n'est pas nécessaire.}
\\We let $\delta<\delta_0$ and we denote
\begin{eqnarray*}
D_\delta':=\Theta^{-1}\left[(-\delta,\delta)\times(1-\delta,1)\right],
\\
{D_{\delta}^{-}}:=\Theta^{-1}\left[(-2\delta,-\delta)\times(1-\delta,1)\right],
\\
{D_{\delta}^{+}}:=\Theta^{-1}\left[(\delta,2\delta)\times(1-\delta,1)\right],
\end{eqnarray*}
so that $D_\delta'$, ${D_{\delta}^{-}}$ and ${D_{\delta}^{+}}$ are Lipschitz domains (see Figure \ref{F3.Fig.LocalMin}). 
 \begin{figure}[h]

\psset{xunit=1.5cm,yunit=1.5cm,algebraic=true,dotstyle=o,dotsize=3pt 0,linewidth=0.8pt,arrowsize=3pt 2,arrowinset=0.25}
\begin{pspicture*}(0.7,-1.6)(10.8,6.9)
\rput{-15.64}(5.61,1.89){\psellipse[fillcolor=lightgray,fillstyle=solid](0,0)(4.81,3.31)}
\rput{-29.37}(5.01,0.07){\psellipse[fillcolor=gray,fillstyle=solid](0,0)(0.52,0.25)}
\rput{36.05}(8.11,0.19){\psellipse[fillcolor= gray,fillstyle=solid](0,0)(0.71,0.38)}
\rput{30.85}(2.73,2.85){\psellipse[fillcolor= gray,fillstyle=solid](0,0)(0.67,0.41)}
\rput{-4.58}(7.13,2.23){\psellipse[fillcolor= gray,fillstyle=solid](0,0)(1.12,0.45)}
\psline[linestyle=dashed,dash=4pt 4pt](3.3,5.17)(3.59,4.06)
\psline[linestyle=dashed,dash=4pt 4pt](6.63,5.05)(6.27,3.83)
\psline[linecolor=red](4.21,4.18)(4.1,5.31)
\psline[linecolor=red](5.86,5.23)(5.65,4.06)
\psarc[linecolor=red](4.7,-1.08){9.63}{79.55}{95.34}
\psarc[linecolor=red](4.61,0.26){5.91}{74.63}{95.74}
\psarc[linestyle=dashed,dash=4pt 4pt](4.61,0.08){7.88}{95.52}{104.42}
\psarc[linestyle=dashed,dash=4pt 4pt](4.61,0.26){5.91}{95.74}{105.03}
\psarc[linestyle=dashed,dash=4pt 4pt](4.61,0.26){5.91}{65.11}{74.63}
\psarc[linestyle=dashed,dash=4pt 4pt](4.62,-1.63){10.46}{73.26}{79.74}
\rput[tl](3.67,5){$D^+_\delta$}
\rput[tl](4.8,5.0){$D'_\delta$}
\rput[tl](5.9,5){$D^-_\delta$}
\rput[tl](9.67,3.65){$h=1-\delta$}
\rput[tl](9.67,6.65){$h=1$}
\rput[tl](6.23,4.4){\scriptsize$\theta=\!-2\delta$}
\rput[tl](3.18,4.4){\scriptsize$\theta=2\delta$}
\rput[lt](3.9,4.50){\parbox{3.01 cm}{ \scriptsize$\theta=\delta$}}
\rput[tl](5.4,4.50){\scriptsize $ \theta=\!-\delta$}
\rput[bl](4.8,0){$\omega_4$}
\rput[bl](7.9,0.07){$\omega_3$}
\rput[bl](2.6,2.79){$\omega_1$}
\rput[bl](6.9,2.2){$\omega_2$}
\rput[tl](9.67,2.8){$\Omega$}
\rput[tl](4.8,1.5){$\dom=\Omega\setminus\cup_i\overline{\omega_i}$}
\rput[bl](5.04,5.35){\darkgray{$ x_0$}}
\rput[bl](4.95,5.27){$\bullet$}
%%%%%%%%%%%%%%%%%
\psline(3.92,4.14)(4.98,3.55)
\psline(4.95,4.18)(4.98,3.55)
\psline(5.96,3.96)(4.98,3.55)
%%%%%%%%%%%%%%%%
\psline(3.65,5.25)(4.56,6.55)
\psline(4.61,5.34)(4.56,6.55)
\psline(6.08,5.19)(4.56,6.55)
%%%%%%%%%%%%%%%%%%
\psline(4.56,6.55)(9.53,6.55)
\psline(4.98,3.55)(9.53,3.55)
%\psdots[dotsize=1pt 0,linecolor=darkgray](5,5.33)
\end{pspicture*}\caption{Decomposition of $\dom$}\label{F3.Fig.LocalMin}
\end{figure}

\subsection{Construction of the test function}
%Remarquons aussi que ${\e^{\imath\theta}}_{|\p\O}:\p\O\rightarrow \mathbb{S}^1$ est globalement définie et $\deg_{\p\O}( \e^{\imath\theta})=1$.

We consider an application (with unknown expression in $D_\delta$) $\psi_t:\dom\rightarrow\C$ ($t>0$ smaller than $\delta$) s.t.
\begin{equation}\label{3.TEST}
\psi_t(x)=\left\{\begin{array}{cl}\d1&\text{in }\dom\setminus D_\delta
%\\
%\d1&\text{sur }\Gamma_\delta\,;
\\
\d\frac{\e^{-\imath\theta}-(1-t\varphi(\theta))}{\e^{-\imath\theta}(1-t\varphi(\theta))-1}&\text{on }\p\O\cap\p D_\delta
%\\
%1&\text{sur }\p\O\setminus\p D_\delta.
\end{array}\right.,
\end{equation}
with $0\leq\varphi\leq1$ a smooth, even and $2\pi$-periodic function  satisfying 
\[
\varphi_{|(-\delta/2,\delta/2)}\equiv1\text{ and }\varphi_{|[-\pi,\pi[\setminus(-\delta,\delta)}\equiv0.
\]
It is clear that $\psi_{t|\p \dom}\in C^\infty(\p \dom)$ and
\begin{equation}\label{3.degint0}
\deg_{\p\omega_i}(\psi_t)=0\text{ for all }i\in\N_N.
\end{equation}

Expanding in Fourier series, we have
\begin{equation}\label{3.decdeM}
\frac{\e^{-\imath\theta}-(1-t\varphi(\theta))}{\e^{-\imath\theta}(1-t\varphi(\theta))-1}=(1-tb_{-1}(t))+t\sum_{k\neq-1}{b_k(t)\e^{-(k+1)\imath\theta}}.
\end{equation}
Noting that the real part of $\d\frac{\e^{-\imath\theta}-(1-t\varphi(\theta))}{\e^{-\imath\theta}(1-t\varphi(\theta))-1}$ is even and the imaginary part is odd, we obtain that $b_k(t)\in\R$ for all $k,t$.

The following lemma is proven in Appendix \ref{3.Apres}
\begin{lem}\label{L3.18}
We denote, for $\e^{\imath\theta}\in \mathbb{S}^1$,
\[
\Psi_t(\e^{\imath\theta})=\frac{\e^{-\imath\theta}-(1-t\varphi(\theta))}{\e^{-\imath\theta}(1-t\varphi(\theta))-1}\phantom{aaaa}\text{ and }\phantom{aaaa}\F_t(\e^{\imath\theta})=\frac{\e^{-\imath\theta}-(1-t)}{\e^{-\imath\theta}(1-t)-1}.
\]Then:
\begin{enumerate}[1)]
\item $|\Psi_t-\F_t|\leq C_\delta\, t\text{ on }\mathbb{S}^1$;
\item $\d \F_t(z)=\frac{\overline{z}-(1-t)}{\overline{z}(1-t)-1}=(1-tc_{-1})+t\sum_{k\neq-1}{c_k(t)\overline{z}^{\,k+1}}$, with 
\[c_k=
\begin{cases} (t-2)(1-t)^k & \text{if $k\geq0$}
\\
0 &\text{if $k\leq-2$}
\\
1&\text{if }k=-1
\end{cases};
\]
\item $|b_k(t)-c_k(t)|\leq C(n,\delta)\left(1+|k|\right)^{-n},\:\forall\,n>0$ with $C(n,\delta)$ independent of $t$ sufficiently small.
%\item  Quitte à considérer $\tilde{b_k}:=-b_k$ à la place de $b_k$, 
%\begin{equation}\label{3.changement}
%M_\lambda(w_t)\leq \delta-2\delta t+2t^2\sum_{\substack{k, l\geq0 \\ k-l > 0 }}{c_kc_l\frac{\sin[(k-l)\delta]}{k-l} \frac{kl}{k+l}}+o(t)
%\end{equation}
\end{enumerate}
\end{lem}

It is easy to see using Lemma \ref{L3.18} that, for $t$ sufficiently small, 
\[
\deg_{\mathbb{S}^1}(\Psi_t)=\deg_{\mathbb{S}^1}(\mathcal{F}_t)=-1.
\]
Using the previous equality and the fact that $\p_\tau\theta>0$ on $\gamma$, we find that
\begin{equation}\label{3.degext0}
\deg_{\p\O}(\psi_t)=-1.
\end{equation}
%{\tt Je ne vois pas dans les calculs qui vont suivre une différence notoire entre le fait de considérer $\psi_t$ ou $\overline{\psi_t}$}

It will be convenient to use $h$ and $\theta$ as a shorthand for $h(x)$ and $\theta(x)$. With these notations, we will look for $\psi_t$ of the form
\begin{eqnarray}\no
\psi_t(x)&=&\tilde{\psi_t}(h,\theta)
\\\label{3.59}
&=&\left\{\begin{array}{cl}\d(1-tf_{-1}(h)b_{-1}(t))+t\sum_{k\neq-1}{b_k(t)f_k(h)\e^{-(k+1)\imath\theta}}&\text{in }D_\delta'
\\
\d\frac{\theta-\delta}{\delta}+\tilde{\psi_t}(h,\delta)\frac{2\delta-\theta}{\delta}&\text{in }{D_{\delta}^{+}}
\\
\d-\frac{\theta+\delta}{\delta}+\tilde{\psi_t}(h,-\delta)\frac{2\delta+\theta}{\delta}&\text{in }{D_{\delta}^{-}}
\end{array}\right..
\end{eqnarray}%avec $\zeta_t\in\mathcal{D}(\R,\R)$, tel que 
%\[\left\{\begin{array}{l}
%\text{Supp}(\zeta_t)\subset(-\delta-Ct^{\beta},\delta+Ct^{\beta}),\,0<\beta<1/2,
%\\
%\zeta\equiv1\text{ dans }(-\delta,\delta),\,|\zeta|\leq1,\,\zeta_t'=\mathcal{O}(t^{-\beta}). 
%\end{array}\right.\]
We impose $f_k(1-\delta)=0$ and $f_k(1)=1$ for $k\in\Z$.
%Par suite $\deg_{\p\O}(\psi_t)=-1$ pour $0<t<1$.

Our aim is to show that for $t>0$ small and appropriate $f_k$'s, the function $\psi_t$ defined by (\ref{3.59}) satisfies (\ref{3.TEST})  and
\begin{equation}\label{3.etoile}
L_\v(\psi_t \e^{\imath \theta},D_\delta)<\pi.
\end{equation}
Here, $L_\v$ is the functional defined in Lemma \ref{L3.16}, so that 
\[
E_\v(\rho \psi_t \e^{\imath \theta},D_\delta)=E_\v(u,D_\delta)+L_\v(\psi_t \e^{\imath \theta},D_\delta).
\]

Then, considering 
\[
\underline{\psi_t}=\begin{cases}\psi_t &\text{if }|\psi_t|\leq2\\\d2\frac{\psi_t}{|\psi_t|}&\text{if }|\psi_t|>2\end{cases}
\]
and setting
\[
\tilde{u}=\left\{\begin{array}{cl}\rho w_t=\underline{\psi_t} u&\text{in }D_\delta\\u&\text{in }\dom\setminus D_\delta\end{array}\right.,
\]
in view of \eqref{3.etoile}, it is straightforward that $\tilde{u}$ satisfies the conclusion of Lemma \ref{L3.15soft}.

\subsection{Upper bound for $L^{}_\v(\cdot,D_\delta)$. An auxiliary problem}
If we let $\tilde{w}:[1-\delta,1]\times[-2\delta,2\delta]$ be s.t. $\tilde{w}(h(x),\theta(x)):=w(x)$, then we have 
\begin{eqnarray*}
|\n w|^2=\sum_i{|\p_i w|^2}&=&\sum_i{|\p_h \tw(h,\theta)\,\p_ih+\p_\theta \tw(h,\theta)\,\p_i\theta|^2}
\\&=&(\rho^4|\p_h \tw(h,\theta)|^2+|\p_\theta \tw(h,\theta)|^2)|\n\theta|^2.
\end{eqnarray*}
%Alors, en faisant le changement de variable donnée par $(\theta,h)$ de $D'_\delta$ dans $\Pi_\delta$, on a
%\[
%\int_{D'_\delta}{\rho^2|\n w|^2\di x}=\frac{1}{d}\int_{\Pi_\delta}{(\rt^4|\p_h\tw|^2d^2+|\p_\theta\tw|^2)\di\theta\di h},
%\]
%\[
%\int_{D'_\delta}{\rho^4(1-|w|^2)^2\di x}=\frac{1}{d}\int_{\Pi_\delta}{\rt^2(1-|\tw|^2)^2\frac{\di\theta\di h}{|\wtt|^2}},
%\]
%\[
%\frac{d^2}{2}\int_{D'_\delta}{|w|^2\rho^2|\n\theta|^2\di x}=\frac{d}{2}\int_{\Pi_\delta}{|\tw|^2\di\theta\di h}.
%\]
Therefore,
\begin{eqnarray}\no
L_\v^{}(w,D_\delta)&=&\frac{1}{2}\int_{D_\delta}\Big\{\left(\rho^4|\p_h \tw(h,\theta)|^2+|\p_\theta \tw(h,\theta)|^2-|\tw(h,\theta)|^2\right)\rho^2|\n\theta|^2+
\\\nonumber
&&\phantom{AAAAAAAAAAAAAAA}+\frac{1}{2\v^2}\rho^4(1-|\tw(h,\theta)|^2)^2\Big\}\di x
\\\nonumber
&\leq&\frac{1}{2}\int_{D_\delta}\Big\{|\p_h \tw(h,\theta)|^2+|\p_\theta \tw(h,\theta)|^2-|\tw(h,\theta)|^2+
\\\label{3.60}&&\phantom{AAAAAAAAAAAAAAA}+\lambda|\e^{\imath \theta}-\tw(h,\theta)|^2\Big\}\rho^2|\n\theta|^2\di x
\\\no&=:&M_\lambda(w,D_\delta),
%\\\no
%&&\phantom{abdadaaaaaaaaaaaaaaaaaaaaaaaaaaaaaada}+\frac{1}{4\v^2}\int_{G''_\delta}{\rho^4(1-|w|^2)^2\di x}
%,\:D_\delta'':=D_\delta\setminus D_\delta'.
\end{eqnarray}
provided that $|w|\leq2$ in $D_\delta$ and $\d\lambda\geq\frac{9}{2\v^2\inf_{D_\delta}{|\n\theta|^2}}$. 

In order to simplify formulas, we will write, in what follows, the second integral in \eqref{3.60} as
\[
\frac{1}{2}\int_{D_\delta}{\left\{|\p_h \tw|^2+|\p_\theta \tw|^2-|\tw|^2+\lambda|\e^{\imath \theta}-\tw|^2\right\}\rho^2|\n\theta|^2\di x}.
\]
The same simplified notation will be implicitly used for similar integrals.
\begin{remark}\label{Remark3.BorneLinfty} If we replace $w$ by $\underline{w}:=\frac{w}{|w|}\min(|w|,2)$, then $M_\lambda$ does not increase. Furthermore replacing $w$ by $\d\underline{w}$ does not affect the Dirichlet condition of (\ref{3.TEST}). Therefore, by replacing $w$ by $\underline{w}$ if necessary, we may assume $|w|\leq2$.
\end{remark}

We next state a lemma which allows us to give a new form of $M_\lambda$.
\begin{lem}\label{L3.17}
Let $f\in C^1(\R,\R)$. Then, for $k\in\Z$, we have
\[
\int_{D'_\delta}{f(h)\cos(k\theta)\rho^2|\n\theta|^2\di x}=\left\{\begin{array}{cl}\d2\delta\int_{1-\delta}^1{f(s)\di s}&\text{if }k= 0\\\d\frac{2\sin(k\delta)}{k}\int_{1-\delta}^1{f(s)\di s}&\text{if }k\neq 0%-\frac{2\pi d_r}{d'^2}\int_{1}^r{f(s)\overline{g}(s)\di s}\text{ sinon.}
\end{array}\right.,
\]
\[
\int_{D^\pm_\delta}{f(h)\,\rho^2|\n\theta|^2\di x}=\delta\int_{1-\delta}^1{f(s)\di s}.
\]
\end{lem}
\begin{proof} This result is easily obtained by noting that the jacobian of the change of variable $x\mapsto(\theta(x),h(x))$ is exactly $\rho^2|\n\theta|^2$.
\end{proof}
For $w=w_t=\psi_t\e^{\imath \theta}$ where $\psi_t$ of the form given by (\ref{3.59}), we have
\begin{eqnarray}\no
M_\lambda(w,D_\delta)&=&\frac{1}{2}\int_{D_\delta}{\left\{|\p_h \tw|^2+|\p_\theta \tw|^2-|\tw|^2+\lambda|\e^{\imath \theta}-\tw|^2\right\}\rho^2|\n\theta|^2\di x}.
%\\\label{3.sanstrop de calcul}
%&=&\frac{1}{2}\int_{D_\delta}{\left(|\p_h \psi|^2+|\p_\theta \psi|^2+\lambda|1-\psi|^2+2\psi\times\p_\theta\psi\right)\rho^2|\n\theta|^2\di x}
%\\\no
%&=&\delta t^2\sum_{k\in\Z}{ b_k^2\phi_k(f_k)}+
\end{eqnarray}
We next rewrite $M_\lambda(w,D_\delta')$. Recalling that for a sequence $\{a_k\}\subset\R$, % is real valued, 
we have
%We claim that for $(a_k)_{k\in\Z}\subset\R$, 
\[
\left|\sum_{k\in\Z}a_k\e^{\imath k\theta}\right|^2=\sum_{k\in\Z}{a_k^2}+2\sum_{\substack{k,l\in\Z,\\k> l}}{a_ka_l\cos[(k-l)\theta]}.
\]
Then we obtain
\begin{eqnarray}\no
M_\lambda(w,D'_\delta)&=&\int_{D'_\delta}\Big\{\frac{t^2}{2}\sum_{k\in\Z}{{b_k} ^2\left[ {f'_k} ^2+{f_k} ^2(k^2+\lambda-1)\right]}%\right.
\\\no
&&\phantom{}-t\sum_{k\neq-1}{b_kf_k(k+1)\cos[(k+1)\theta]}
\\\no
&&\phantom{}-t^2\sum_{k\neq-1}b_{-1}b_k[f_{-1}'f_k'-f_{-1}f_k(k-\lambda+1)]\cos[(k+1)\theta]
\\\label{3.85}
&&\phantom{}+t^2\sum_{\substack{k, l\neq-1 \\ k-l > 0 }}{b_kb_l[f_k'f_l'+(kl+\lambda-1)f_kf_l]\cos[(k-l)\theta]}\Big\}\rho^2|\n\theta|^2.
\end{eqnarray}
Using Lemma \ref{L3.17} and (\ref{3.85}), we have 
\begin{eqnarray}\no
M_\lambda(w ,D'_\delta)&=&\delta t^2\sum_{k\in\Z}{b_k^2\phi_k(f_k)}-2t\sum_{k\neq-1}{b_k\sin[(k+1)\delta]\,\int_{1-\delta}^1{f_k}}
\\\no
&&-2t^2\sum_{k\neq-1}{b_{-1}b_k\frac{\sin[(k+1)\delta]}{k+1}\int_{1-\delta}^1{\left\{f_{-1}'f_k'-(k-\lambda+1)f_{-1}f_k\right\}}}
\\\label{3.totala}
&&+2t^2\sum_{\substack{k, l\neq-1 \\ k-l > 0}}{b_kb_l\frac{\sin[(k-l)\delta]}{k-l}\int_{1-\delta}^1{\left\{f_k'f_l'+(kl+\lambda-1)f_kf_l\right\}}}
\\\label{3.totala1}&=&R_\lambda(w)-2t\sum_{k\neq-1}{b_k\sin[(k+1)\delta]\,\int_{1-\delta}^1{f_k}}.
\end{eqnarray}
with
\begin{eqnarray*}
R_\lambda(w)&=&\delta t^2\sum_{k\in\Z}{b_k^2\phi_k(f_k)}
\\&&-2t^2\sum_{k\neq-1}{b_{-1}b_k\frac{\sin[(k+1)\delta]}{k+1}\int_{1-\delta}^1{\left\{f_{-1}'f_k'-(k-\lambda+1)f_{-1}f_k \right\}}}
\\\no
&&+2t^2\sum_{\substack{k, l\neq-1 \\ k-l > 0}}{b_kb_l\frac{\sin[(k-l)\delta]}{k-l}\int_{1-\delta}^1{\left\{f_k'f_l'+(kl+\lambda-1)f_kf_l\right\}}},
\end{eqnarray*}

\[
\phi_k(f)=\int_{1-\delta}^1{\left\{f'^2+\alpha_k^2f^2\right\}}
\]
and
\[
\d\alpha_k=\sqrt{k^2+\lambda-1}.
\]
We next establish a similar identity for $M_\lambda(w_t,D_\delta^\pm)$. Using (\ref{3.59}), we have
\begin{eqnarray}\no
M_\lambda(w_t,D_\delta^\pm)&=&\frac{1}{2}\int_{D^\pm_\delta}{\left\{|\p_h \tw(h,\theta)|^2+|\p_\theta \tw(h,\theta)|^2-|w|^2+\lambda|\e^{\imath \theta}-w|^2\right\}\rho^2|\n\theta|^2}
\\\no
&=&\frac{1}{2}\int_{D^\pm_\delta}\bigg\{|\p_h \tilde{\psi_t}(h,\pm\delta)|^2\left(\frac{2\delta\mp\theta}{\delta}\right)^2
\\\no%\label{3.gpm-decomposition}
&&+\delta^{-2}(1+\lambda\left(2\delta\mp\theta\right)^2)|\tilde{\psi_t}(h,\pm\delta)-1|^2\mp2\delta^{-1}\text{Im }\tilde{\psi_t}(h,\pm\delta)\bigg\}\rho^2|\n\theta|^2
\\\no
&=&\frac{1}{2\delta^2}\int_{D^\pm_\delta}\Big\{|\p_h \tilde{\psi_t}(h,\pm\delta)|^2\left(2\delta\mp\theta\right)^2
\\\no
&&\phantom{aaaaaaa}+(1+\lambda\left(2\delta\mp\theta\right)^2)|\tilde{\psi_t}(h,\pm\delta)-1|^2\Big\}\rho^2|\n\theta|^2
\\\label{3.gpm-decomposition}
&&\phantom{aaaaaaaaqqqqqqqaaaaaa}+t\sum_{k\neq-1}{b_k(t)\sin[(k+1)\delta]\int_{1-\delta}^1f_k}.
\end{eqnarray}
Here, $\text{Im }\psi$ denotes the imaginary part of $\psi$.
To obtain (\ref{3.gpm-decomposition}), we used the identity
\[
|\p_\theta(\psi\e^{\imath\theta})|^2=|\p_\theta\psi|^2+|\psi|^2+2\psi\times\p_\theta\psi.
\]
%From (\ref{3.totala1}) and (\ref{3.gpm-decomposition}), we have
%\begin{equation}\label{3.final decres}
%M_\lambda(w_t,D_\delta)=
%\end{equation}
\subsection{Choice of $w=\psi_t{\rm e}^{\imath\theta}$}
We take %(en posant $\alpha=\alpha_k$), 
\begin{equation}\label{3.fk}
f_k(h)=\frac{\e^{\alpha_k( h-1)}}{1-\e^{-2\alpha_k\delta}}+\frac{\e^{-\alpha_k (h-1)}}{1-\e^{2\alpha_k\delta}}.
\end{equation}
With this choice, by direct computations we have
\begin{equation}\label{3.est1}
\phi_k(f_k)=\alpha_k\left(1+\frac{2}{\e^{2\alpha_k\delta}-1}\right),
\end{equation}
\begin{equation}\label{3.est2}
\int_{1-\delta}^1f_k=\frac{1}{\alpha_k}\left(1-\frac{2}{\e^{\alpha_k\delta}+1}\right)
\end{equation}and for $k,l\in\Z$ s.t. $k\neq\pm l$,
\begin{equation}\label{3.est3}
\int_{1-\delta}^1f_kf_l=\frac{1-\e^{-2(\alpha_k+\alpha_l)\delta}}{(\alpha_k+\alpha_l)(1-\e^{-2\alpha_k\delta})(1-\e^{-2\alpha_l\delta})}-\frac{1-\e^{-2(\alpha_k-\alpha_l)\delta}}{(\alpha_k-\alpha_l)(1-\e^{-2\alpha_k\delta})(\e^{2\alpha_l\delta}-1)},
\end{equation}
\begin{eqnarray}\no
\frac{1}{\alpha_k\alpha_l}\int_{1-\delta}^1f'_kf'_l&=&\frac{1-\e^{-2(\alpha_k+\alpha_l)\delta}}{(\alpha_k+\alpha_l)(1-\e^{-2\alpha_k\delta})(1-\e^{-2\alpha_l\delta})}
\\\label{3.est4}
&&\phantom{aaaaaaaa}+\frac{1-\e^{-2(\alpha_k-\alpha_l)\delta}}{(\alpha_k-\alpha_l)(1-\e^{-2\alpha_k\delta})(\e^{2\alpha_l\delta}-1)}.%,\,k\neq\pm l.
\end{eqnarray}

Using \eqref{3.totala1}—\eqref{3.est4}, we may obtain the following estimate, whose proof is postponed to Appendix \ref{3.Apres}.
\begin{lem}\label{L3.simplification}We have
\begin{equation}\label{3.changement}
M_\lambda(w_t,D_\delta)\leq \delta-2\delta t+4t^2\sum_{k>l>0}{c_kc_l\frac{\sin[(k-l)\delta]}{k-l} \frac{kl}{k+l}}+o(t).
\end{equation}
\end{lem}
%Afin d'estimer $M_\lambda(w_t)$, on définit $\underline{w}_t(\theta,h)(\cdot)\in W_t$ par%(par le lemme \ref{L3.18})
%\[
%\underline{w}_t=\left[(1-tc_{-1}f_{-1}(h))+t\sum_{k\neq-1}{c_k(t)f_k(h)\e^{-\imath(k+1)\theta}}\right]\e^{\imath \theta}.
%\] 
\subsection{End of the proof of Lemma \ref{L3.15soft}}
We denote
\begin{equation}\label{3.def.S.delta}
S(\delta,t):=\sum_{k>l>0}{c_kc_l\frac{\sin[(k-l)\delta]}{k-l} \frac{kl}{k+l}}.%=\frac{(2-t)^2}{4}\sum_{\substack{k, l>0 \\ k-l > 0 }}{(1-t)^{k+l}\frac{\sin[(k-l)\delta]}{k-l} \left(k+l-\frac{(k-l)^2}{k+l}\right)}.
\end{equation}
Setting $n=k-l$ and noting that $\d\frac{(n+l)l}{n+2l}=\frac{l}{2}+\frac{ln}{2(n+2l)}$, we have
\begin{eqnarray*}
\frac{2}{(t-2)^2}S(\delta,t)&=&\sum_{n>0}{(1-t)^{n}\frac{\sin (n\delta)}{n}}\sum_{l>0}{l(1-t)^{2l}}%+\sum_{n>0}{(1-t)^n\sin (n\delta)}\sum_{l>0}{(1-t)^{2l}}
%\\
%&&\phantom{aaaaaavvvvvvvvvvvvaaaaaaaaaaaaaaaaaaaa}
+\sum_{n,l>0}{(1-t)^{n+2l}\sin (n\delta)\frac{l}{n+2l}}.
\end{eqnarray*}
Here, we have used the explicit formulae for the $c_k$'s, given by Lemma \ref{L3.18}.

Using Appendix \ref{3.Apsoft} (see Appendix \ref{3.AP1}) we find that for $0<t<\delta$, we have
%\[
%\sum_{l>0}{l(1-t)^{2l}}=\frac{(1-t)^2}{t^2\left(2-t\right)^2},
%\]

%\[
%\sum_{n>0}{(1-t)^{n}\sin (n\delta)}=\frac{(1-t)\sin \delta}{1-2(1-t)\cos\delta+(1-t)^2},
%\]
%\[
%\sum_{n>0}{(1-t)^{n}\frac{\sin n\delta}{n}}=\arctan \frac{1-t-\cos\delta}{\sin\delta}+\arctan \frac{\cos\delta}{\sin\delta},
%\]
%\begin{eqnarray*}
%\sum_{n,l>0}{(1-t)^{n+2l}\sin n\delta\frac{l}{n+2l}}&=&\frac{1-t+\cos\delta}{4\sin\delta t(2-t)}-\frac{1}{4\sin^2\delta}\arctan\frac{1-t-\cos\delta}{\sin\delta}+\text{Cst}(\delta).
%\left|\sum_{n,l>0}{(1-t)^{n+2l}\sin n\delta\frac{l}{n+2l}}\right|&\leq&\sum_{n,l>0}{(1-t)^{n+2l}\frac{l}{n+2l}}
%\\
%&\leq&\frac{A}{2-t}+\frac{B}{t}+C\ln(2-t)+D\ln t+E\arctan\frac{1-t-\cos\delta}{\sin\delta}+F,
%\end{eqnarray*}
%On voit alors que pour $\delta>t>0$ assez petit, $\d\sum_{n,l>0}{(1-t)^{n+2l}\sin n\delta\frac{l}{n+2l}}>0$, par suite, 
\begin{eqnarray}\nonumber
S(\delta,t)&=&\frac{(1-t)^2}{2t^2}\left[\arctan \left(\frac{1-t-\cos\delta}{\sin\delta}\right)+\arctan\left( \frac{\cos\delta}{\sin\delta}\right)\right]\\\label{3.part3.1}&&\phantom{aaaaaaaaaaasssssssssssssss}+\frac{(1-t+\cos\delta)(2-t)}{8t\sin\delta }+\mathcal{O}(1).
\end{eqnarray}
We note that
\begin{eqnarray}\nonumber
\arctan \left(\frac{1-t-\cos\delta}{\sin\delta}\right)&=&\arctan\left( \frac{1-\cos\delta}{\sin\delta}\right)-\frac{t\sin\delta}{2(1-\cos\delta)}+\mathcal{O}(t^2)\\\label{3.part3.2}&=&\frac{\delta}{2}-\frac{t\sin\delta}{2(1-\cos\delta)}+\mathcal{O}(t^2)
\end{eqnarray}
and
\begin{equation}\label{3.part3.21}
\arctan \left(\frac{\cos\delta}{\sin\delta}\right)=\frac{\pi}{2}-\delta.
\end{equation}
From (\ref{3.part3.1})—(\ref{3.part3.21}) we infer
\begin{equation}\label{3.part3.22bis}
S(\delta,t)\leq\frac{1}{4t^2}(\pi-\delta)+\frac{1}{ t}\left[\frac{(1-t+\cos\delta)(2-t)}{8\sin\delta }-\frac{\sin\delta}{4(1-\cos\delta)}\right]+\mathcal{O}(1)
\end{equation}
with
\begin{equation}\label{3.part3.22}
\frac{(1-t+\cos\delta)(2-t)}{8\sin\delta }-\frac{\sin\delta}{4(1-\cos\delta)}<\frac{(1+\cos\delta)}{4\sin\delta }-\frac{\sin\delta}{4(1-\cos\delta)}=0.
\end{equation}
From (\ref{3.changement}),
\begin{equation}\label{3.sumupbis}
M_\lambda(w_t,D_\delta)\leq\delta-2\delta t+4t^2 S(\delta,t)+o(t).
\end{equation}
Using \eqref{3.part3.22bis} and \eqref{3.part3.22},
\begin{equation}\label{3.sumup}
4t^2 S(\delta,t)\leq\pi-\delta+o(t).
\end{equation}
Finally, we have by combining \eqref{3.sumupbis} with \eqref{3.sumup},
\begin{equation}\label{3.resultat.final}
M_\lambda(w_t,D_\delta)\leq\pi-2\delta t+o(t)<\pi\text{ for }t\text{ small.}
\end{equation}

We conclude that for $t$ sufficiently small, $L_\v^d(\underline{w}_t,D_\delta)<\pi$. 

\subsection{Conclusion}
$\tilde{u}:=\underline{\psi} u$, with $\underline{\psi}=\d\psi_t\frac{\min(|\psi_t|,2) }{|\psi_t|}$, satisfies the desired properties \emph{i.e.}:
\begin{description}
\item[$\bullet$] $E_\v(\tilde{u})<E_\v(u)+\pi$ (by (\ref{3.60}) and (\ref{3.resultat.final})) ;
\item[$\bullet$]$\tilde{u}\in \J^{\db }_{{\bf p},q-1}$ (by (\ref{3.83}), (\ref{3.degint0}) and (\ref{3.degext0})).
\end{description}
\subsection{A direct consequence of Lemma \ref{L3.15soft}}
By applying Lemma \ref{L3.15soft} and next Lemma \ref{P3.bougedegmult}, one may easily obtain the following
\begin{cor}\label{C3.DirectConsequenceOfMainTool}
Let $u\in \J_{{\bf p},q}$ be a solution of (\ref{3.8}), (\ref{3.8'}). 

Assume that 
\begin{equation}\nonumber%\label{3.deg}
\ab_j(u)\in(d_j-\frac{1}{3},d_j+\frac{1}{3}),\:\forall\,j\in\N_N.
\end{equation}
Assume that there are  $i_0\in \{0,...,N\}$ and $x_0\in\p\omega_{i_0}$ s.t. $ u\times\p_\tau u(x_0)>0$. 

Then for all $\delta=(\delta_1,...,\delta_N,\delta_0)\in\Z^{N+1}$ s.t. $\delta_{i_0}>0$, there is $\tilde{u}_\delta\in\J_{({\bf p},q)-\delta}$ s.t. 
\begin{equation}\no
E_\v(\tilde{u}_\delta)<E_\v(u)+\pi\sum_i|\delta_i|
\end{equation}
and
\[
\ab_j(\tilde{u}_\delta)\in(d_j-\frac{1}{3},d_j+\frac{1}{3}),\,\forall j\in\N_N.
\]
\end{cor}

\section{Proof of Theorem \ref{T3.Existence}}
The energy estimate is obtained from Lemmas \ref{L3.5} and \ref{L3.bornesup}.

%The proof is made by induction on 
%\[
%K=\textsf{\ae}((\db,d),({\bf p},q))=|d_1-p_1|+...+|d_N-p_N|+|d-q|\geq0.
%\]
We call $({\bf p},q,\db)$ a good configuration of degrees if
\[
({\bf p},q,\db)\in\Z^N\times\Z\times(\N^*)^N,\,p_i\leq d_i\text{ and }q\leq\sum_i{d_i}=:d.
\]
We first prove Theorem \ref{T3.Existence} when 
\[
\textsf{\ae}((\db,d),({\bf p},q))=|d_1-p_1|+...+|d_N-p_N|+|d-q|=0\Leftrightarrow {\bf p}=\db\text{ and }q=d.  
\]
For $\v>0$, let  $(u^\v_n)_n$ be a minimizing sequence of $E_\v$ in $\J_{\db,d}^\db$. For $\v<\v_{\ref{3.eps4}}({\bf d},d,{\bf d})$, up to subsequence, using Proposition \ref{P3.min}, $u^\v_n\rightarrow u_\v$ weakly in $H^1$ and strongly in $L^4$ and $u_\v$ is a (global) minimizer of $E_\v$ in $\J_{\deg(u_\v,\dom)}^\db$. 
 
Applying Lemmas \ref{L3.6} and \ref{L3.5}, for $\v<\v_{\ref{3.eps2}}(\db)\leq\v_{\ref{3.eps4}}$ (here, $\v_{\ref{3.eps2}}$ is s.t. the $o_\v(1)$ of Lemma \ref{L3.5} is lower than $\d\frac{\pi}{2}$), 
\begin{eqnarray}\no
I_0(\db ,\dom)&\geq&E_\v(u_\v)+\pi\tae(\deg(u_\v,\dom),({\bf d},d))
\\\no
&\geq&I_0(\db ,\dom)-\frac{\pi}{2}+2\pi\tae(\deg(u_\v,\dom),({\bf d},d)).
\end{eqnarray}
It follows, $\d\tae(\deg(u_\v,\dom),({\bf d},d))\leq\frac{1}{4}$ which implies $u_\v\in\J_{\db,d}^\db$.

%Assume that Theorem \ref{T3.Existence} is true for all configurations $({\bf p},q,\db)$ s.t. 
%\[
%0\leq \tae(({\bf p},q),(\db,d))\leq K.
% \]
We now prove (following the same strategy) Theorem \ref{T3.Existence} for a good configuration $({\bf p},{q},\db)$ s.t. 
\[
\tae(({\bf p},{q}),(\db,d))>0.
 \]
For $\v>0$ consider $(u_n^\v)_n$ a minimizing sequence of $E_\v$ in $\J_{{\bf p},{q}}^\db$.

For $\v<\v_{\ref{3.eps4}}({\bf p},{q},\db)$, up to subsequence, using Proposition \ref{P3.min}, $u^\v_n\rightarrow u_\v$ weakly in $H^1$ and strongly in $L^4$ and $u_\v$ is a (global) minimizer of $E_\v$ in $\J_{\deg(u_\v,\dom)}^\db$. 
 %Let $(\tilde{\bf p},\tilde{q},\db)$ be a good configuration s.t. $\tae((\tilde{\bf p},\tilde{q}),(\db,d))=K+1$. Then there is some $i\in\{0,...,N\}$ s.t. $((\tilde{\bf p},\tilde{q})+{\bf e}_i,\db)$ is a good configuration and $\tae((\tilde{\bf p},\tilde{q})+{\bf e}_i,(\db,d))=K$.
 
 %Without loss of generality, we may assume that $i=0$. Set ${\bf p}=\tilde{\bf p}$ and $\tilde{q}+1=q$. By induction hypothesis, Theorem \ref{T3.Existence} holds for $({\bf p},q,\db)$.

%Let $\v_{\ref{3.eps2}}({\bf p},q,\db)>0$ be defined by induction hypothese and $(u_\v)_{\v_{\ref{3.eps2}}>\v>0}$ be a family of (global) minimizers of $(E_\v)_{\v_{\ref{3.eps2}}>\v>0}$ in $\J_{{\bf p},q}^\db$. 

Let $\Lambda:=I_0(\db,\dom)+\tae(({\bf p},{q}),(\db,d))\pi+1$, by Lemma \ref{L3.deg}, for $\v<\v_{\ref{3.eps5}}({\bf p},q,\db,\Lambda)$, there is $x_\v^0\in\p\O$ s.t. $(u_\v\times\p_\tau u_\v)(x_\v^0)>0$.% on $\p\O$.

The third assertion in Proposition \ref{P3.1} and the energy bound give the existence of  $0<\v_{\ref{3.eps2}}'({\bf p},q,\db,\Lambda)<\v_{\ref{3.eps5}}({\bf p},q,\db,\Lambda)$ s.t. for $0<\v<\v_{\ref{3.eps2}}'$,
\[
\ab_i(u_\v)\in(d_i-\frac{1}{3},d_i+\frac{1}{3}).
\]
%Using Lemma \ref{L3.15soft}, for $\v<\v_{\ref{3.eps2}}'$, we have the existence of $\tilde{u_\v}\in\J_{{\bf p},q-1}^\db$ s.t.
%\[
%m_\v({\bf p},q,\db)+ \pi =E_\v(u_\v)+ \pi >E_\v(\tilde{u_\v})\geq m_\v({\bf p},q-1,\db).
%\]

%Using \eqref{3.supmv}, we have,%the energy bound of the induction hypothesis,
%\[
%m_\v({\bf p},q-1,\db)< I_0(\db,\dom)+(K+1)\pi.%\left(d_1-p_1+...+d_N-p_N+d-(q-1)\right).
%\]

%Let $(u^\v_n)_n$ be a minimizing sequence of $E_\v$ in $\J_{{\bf p},q-1}^\db$. Using Proposition \ref{P3.min}, for $\v<\v_{\ref{3.eps4}}$, up to subsequence, $u_n^\v\rightarrow u^\v$ weakly in $H^1$ and strongly in $L^2$ with $u^\v$ which is a global minimizer of $E_\v$ in $\J^\db_{\deg(u^\v,\dom)}$.

%Par convergence forte dans $L^2$, $\ab_i(u)=d_i$ et par convergence faible $E_\v(u)\leq m_\v(\bfp,q,\db)+\pi$.

%It suffices to show that, for $\v$ sufficiently small, $u^\v\in \J_{{\bf p},q-1}^\db$. The argument is similar to the one used for $K=0$. By strong $L^2$ convergence, $\ab_i(u^\v)\in[d_i-1/3,d_i+1/3]$. 

Fix $\v_{\ref{3.eps2}}'({\bf p},q,{\bf d})>\v_{\ref{3.eps2}}({\bf p},q,{\bf d})>0$ s.t. the $o_\v(1)$  in Lemma \ref{L3.5} is lower than $\d\frac{\pi}{2}$ (here $\v_{\ref{3.eps5}}$ is defined in Lemma \ref{L3.deg}).

Using Lemmas \ref{L3.6}, \ref{L3.5} and \ref{L3.bornesup}, we have for $\v<\v_2$
\begin{eqnarray*}
 I_0(\db,\dom)+\pi\tae(({\bf p},{q}),(\db,d))&\geq&\liminf E_\v(u_n^\v)\:(\text{by Lemma  \ref{L3.bornesup} and the definition of }(u^\v_n)_n)%m_\v(\bfp,q-1,\db)
% \\&\geq&\liminf E_\v(u^\v_n)\:\:\:(\text{by the definition of }(u^\v_n)_n)%&\geq&E_\v(u)
\\\no
&\geq&E_\v(u_\v)+\pi\tae\left(({\bf p},q),\deg(u_\v,\dom)\right)\:\:\:(\text{Lemma \ref{L3.6}})
\\\no
&\geq&I_0(\db,\dom)+\pi\left[\tae\left(({\bf p},q),\deg(u^\v,\dom)\right)\right.
\\&&\phantom{aaaaa}\left.+\tae\left(({\bf d},d),\deg(u^\v,\dom)\right)\right]-\frac{\pi}{2}\:\:(\text{Lemma \ref{L3.5}})
%\\
%&\geq&I_0(\db,\dom)+\pi\left(\tae\left(({\bf p},q-1),(\db,d)\right)-\frac{1}{2}\right)\:(\text{by the triangle inequality})
%\\
%&\geq&I_0(\db,\dom)+\left(K+\frac{1}{2}\right)\pi.
\end{eqnarray*}
It follows that
\begin{equation}\label{3.Ledegreestdanslebonintervalle}
\tae\left(({\bf p},q),\deg(u^\v,\dom)\right)+\tae\left(({\bf d},d),\deg(u^\v,\dom)\right)=\tae(({\bf p},{q}),(\db,d)).%\sum_{i\in\N_N}{|p_i-\deg_{\p\omega_i}(u^\v)|}+|q-1-\deg_{\p\Omega}(u^\v)|+\sum_{i\in\N_N}{|d_i-\deg_{\p\omega_i}(u^\v)|}+|d-\deg_{\p\Omega}(u^\v)|= K+1
\end{equation}
Thus%Since $\tae\left(({\bf p},q-1),({\bf d},d)\right)=K+1$, we must have
\begin{equation}\nonumber%\label{3.finiteset}
p_i\leq\deg_{\p\omega_i}(u^\v)\leq d_i\text{ and }q\leq\deg_{\p\O}(u^\v)\leq d.
\end{equation}
Assume that there is $\v<\v_2$ s.t. $u_\v\notin\J_{{\bf p},q}^{\bf d}$. Then from Lemma \ref{L3.deg} and \eqref{3.Ledegreestdanslebonintervalle}, one may apply Corollary \ref{C3.DirectConsequenceOfMainTool} to obtain the existence of $\tilde{u}_\v\in\J_{{\bf p},q}^{\bf d}$ s.t.
\[
m_\v({\bf p},q,{\bf d})\leq E_\v(\tilde{u}_\v)<E_\v(u_\v)+\pi\tae\left(({\bf p},q),\deg(u_\v,\dom)\right)\leq \liminf E_\v(u_n^\v)=m_\v({\bf p},q,{\bf d})
\]
which is a contradiction.

Thus for $\v<\v_2$, $u_\v\in\J_{{\bf p},q}^{\bf d}$ and consequently $u_\v$ is a minimizer of $E_\v$ in $\J_{{\bf p},q}^{\bf d}$.

\appendix
\renewcommand{\theequation}% 
{\thesection \Alph{equation}} 
\numberwithin{equation}{section}
%\begin{subappendices}
\section{Results used in the proof of Lemma \ref{L3.15soft}}\label{3.Apsoft}

\subsection{Power series expansions}\label{3.AP1}
For $X\in\C$, $|X|<1$, we have
\begin{equation}\label{3A.1}
\sum_{k\geq1}{\frac{|X|^k}{k}}=-\ln(1-|X|),
\end{equation}
\begin{equation}\label{3A.2}
\sum_{k\geq0}{X^k}=\frac{1}{1-X},
\end{equation}
\begin{equation}\label{3A.3}
\sum_{k\geq1}{kX^k}=\frac{X}{(1-X)^2},
\end{equation}
\begin{equation}\label{3A.4}
\sum_{k>0}{\sin (k\delta) X^{k}}=\frac{X\sin \delta}{1-2X\cos\delta+X^2},
\end{equation}
\begin{equation}\label{3A.5}
\sum_{k>0}{\frac{\sin (k\delta)}{k}X^{k}}=\arctan \left(\frac{X-\cos\delta}{\sin\delta}\right)+\arctan \left(\frac{\cos\delta}{\sin\delta}\right),
\end{equation}
\begin{equation}\label{3A.6}
\sum_{n,l>0}{\sin (n\delta)\frac{l}{n+2l}X^{n+2l}}=\frac{X+\cos\delta}{4(1-X^2)\sin\delta}-\frac{1}{4\sin^2\delta}\arctan\left(\frac{X-\cos\delta}{\sin\delta}\right)+\text{Cst}(\delta).
\end{equation}
{\bf Proof: }The first four identities are classical.  We sketch the argument that leads to (\ref{3A.5}) and (\ref{3A.6}). The identity (\ref{3A.5}) follows from (\ref{3A.4}) by integration. 

We next prove (\ref{3A.6}). Let 
\[
f(X)=\sum_{n,l>0}{\sin (n\delta)\frac{l}{n+2l}X^{n+2l}}.
\]
On the one hand, by (\ref{3A.3}), (\ref{3A.4}),
\[
f'(X)=\frac{1}{X}\sum_{n>0}{\sin (n\delta) X^{n}}\sum_{l>0}{lX^{2l}}=\frac{X^2\sin \delta}{(1-X^2)^2(1-2X\cos\delta+X^2)}.
\]
On the other hand
\[
\frac{\di}{\di X}\left(\frac{X+\cos\delta}{4\sin\delta(1-X^2)}-\frac{1}{4\sin^2\delta}\arctan\frac{X-\cos\delta}{\sin\delta}\right)=\frac{X^2\sin \delta}{(1-X^2)^2(1-2X\cos\delta+X^2)}.
\]
\subsection{Estimates for $f_k$ and $\alpha_k$}\label{3.AP2}
Recall that we defined, in section \ref{3.sectionPREUVE}, $f_k$ and $\alpha_k$ by
\begin{equation}\no%\tag{\ref{3.fk}}
f_k(h)=\frac{\e^{\alpha_k( h-1)}}{1-\e^{-2\alpha_k\delta}}+\frac{\e^{-\alpha_k (h-1)}}{1-\e^{2\alpha_k\delta}},
\end{equation}
\[
\d\alpha_k=\sqrt{k^2+\lambda-1}.
\]
In this part, we prove the following inequalities:
\begin{equation}\label{3A.0}
\alpha_k= |k|+\mathcal{O}\left(\frac{1}{|k|+1}\right),
\end{equation}
\begin{equation}\label{3A.7}
|f_k(h)-\e^{-|k|(1-h)}|\leq\frac{C}{k^2},\text{ with }C\text{ independent of }k\in\Z^*,\,h\in(1-\delta,1),
\end{equation}
\begin{equation}\label{3A.8}
|f'_k(h)-|k|\e^{-|k|(1-h)}|\leq\frac{C}{|k|},\text{ with }C\text{ independent of }k\in\Z^*,\,h\in(1-\delta,1).
\end{equation}
{\bf Proof: }The first assertion is obtained using a Taylor expansion.

Let $g_h(u)=\e^{u( h-1)}$, we have 
\begin{eqnarray*}
|f_k(h)-\e^{-|k|(1-h)}|&\leq&|g_h(\alpha_k)-g_h(|k|)|+\frac{C}{k^2}
\\&\leq& \sup_{(|k|,\alpha_k)}|g_h'(u)||\alpha_k-|k||+\frac{C}{k^2}\leq\frac{1}{\e k}\frac{1}{2 k}+\frac{C}{k^2}\leq\frac{C}{k^2}.
\end{eqnarray*}
The proof of \eqref{3A.8} is similar, one uses $\tilde{g}_h(u)=u\e^{u(h-1)}$ instead of $g_h$

\subsection{Further estimates on $f_k$ and $\alpha_k$}\label{3.AP3}
We have
\begin{eqnarray}\no
0\leq\int_{1-\delta}^1{\left\{{f'_k}^2-\alpha_k^2{f_k}^2\right\}}&\leq&\int_{1-\delta}^1{\left\{{f'_k}^2-k^2{f_k}^2\right\}}
\\\label{3A.9}
&\leq&\frac{C}{|k|+1},\text{ with }C\text{ independent of }k\in\Z,
\end{eqnarray}
\begin{equation}\label{3A.10}
\left|\int_{1-\delta}^1{f_kf_l}\right|\leq\frac{C}{\max(|k|,|l|)},\text{ with }C\text{ independent of }k,l\in\Z,\,\text{ s.t. }|k|\neq|l|,
\end{equation}
\begin{equation}\label{3A.11}
\left|\int_{1-\delta}^1{f'_kf'_l}\right|\leq C\left(\min(|k|,|l|)+1\right),\text{ with }C\text{ independent of $k,l\in\Z$, s.t. $|k|\neq|l|$}.
\end{equation}
{\bf Proof: }Actually \eqref{3A.10}, \eqref{3A.11} still hold when $|k|=|l|$, but this will not used in the proof of Lemma \ref{L3.15soft} and requires a separate argument.

Since $\alpha_k\geq |k|$,
\[
\int_{1-\delta}^1{\left\{{f'_k}^2-\alpha_k^2{f_k}^2\right\}}\leq\int_{1-\delta}^1{\left\{{f'_k}^2-k^2{f_k}^2\right\}}.
\]
By direct computations,
\[
0\leq\int_{1-\delta}^1{\left\{{f'_k}^2-\alpha_k^2{f_k}^2\right\}}=\frac{4\delta\alpha_k^2}{(1-\e^{-2\alpha_k\delta})(\e^{2\alpha_k\delta}-1)}\leq\frac{C(\delta,n)}{k^n},\phantom{aaaaa}\forall n\in\N^*,
\]
\[
\int_{1-\delta}^1{\left\{{f'_k}^2-k^2{f_k}^2\right\}}=\int_{1-\delta}^1{\left\{{f'_k}^2-\alpha_k ^2{f_k}^2\right\}}+(\lambda-1)\int_{1-\delta}^1{{f_k}^2},
\]
\[
\int_{1-\delta}^1{{f_k}^2}=\frac{1}{2\alpha_k}\left(\frac{1}{1-\e^{-2\alpha_k\delta}}-\frac{1}{1-\e^{2\alpha_k\delta}}\right)+\mathcal{O}\left(\frac{1}{|k|+1}\right)=\mathcal{O}\left(\frac{1}{|k|+1}\right).%-\frac{\e^{-2\alpha_k\delta}}{(1-\e^{-2\alpha_k\delta})^2}+\frac{\e^{2\alpha_k\delta}}{(1-\e^{2\alpha_k\delta})^2}\right)+\frac{2\delta}{(1-\e^{2\alpha_k\delta})(1-\e^{2\alpha_k\delta})}
\]Which proves \eqref{3A.9}.
\\For $|k|\neq|l|$, we have
\begin{eqnarray}\no%\label{3.est3}
\left|\int_{1-\delta}^1f_kf_l\right|&=&\left|\frac{1-\e^{-2(\alpha_k+\alpha_l)\delta}}{(\alpha_k+\alpha_l)(1-\e^{-2\alpha_k\delta})(1-\e^{-2\alpha_l\delta})}-\frac{1-\e^{-2(\alpha_k-\alpha_l)\delta}}{(\alpha_k-\alpha_l)(1-\e^{-2\alpha_k\delta})(\e^{2\alpha_l\delta}-1)}\right|
\\\no
&\leq&\frac{C}{\max(|k|,|l|)}+\left|\frac{1-\e^{-2(\alpha_k-\alpha_l)\delta}}{(\alpha_k-\alpha_l)(1-\e^{-2\alpha_k\delta})(\e^{2\alpha_l\delta}-1)}\right|.
\end{eqnarray}
We assume that $|k|>|l|$ and we consider the two following cases: $\alpha_l<\alpha_k\leq2\alpha_l$ and $\alpha_k>2\alpha_l$. Noting that $\frac{1-\e^{-2x\delta}}{x}$ is bounded for $x\in\R^*_+$, we have
\[
\left|\frac{1-\e^{-2(\alpha_k-\alpha_l)\delta}}{(\alpha_k-\alpha_l)(1-\e^{-2\alpha_k\delta})(\e^{2\alpha_l\delta}-1)}\right|\leq\frac{C}{\e^{2\alpha_l\delta}}\leq \frac{C}{\max(|k|,|l|)}\text{ if }\alpha_l<\alpha_k\leq2\alpha_l,
\]
\[
\left|\frac{1-\e^{-2(\alpha_k-\alpha_l)\delta}}{(\alpha_k-\alpha_l)(1-\e^{-2\alpha_k\delta})(\e^{2\alpha_l\delta}-1)}\right|\leq\frac{C}{\alpha_k-\alpha_l}\leq \frac{C}{\max(|k|,|l|)}\text{ if }\alpha_k>2\alpha_l.
\]This proves \eqref{3A.10}.
\\For $|k|\neq|l|$, 
\begin{equation}\no%\label{3.est4}
\int_{1-\delta}^1f'_kf'_l=\frac{\alpha_k\alpha_l\left(1-\e^{-2(\alpha_k+\alpha_l)\delta}\right)}{(\alpha_k+\alpha_l)(1-\e^{-2\alpha_k\delta})(1-\e^{-2\alpha_l\delta})}+\frac{\alpha_k\alpha_l\left(1-\e^{-2(\alpha_k-\alpha_l)\delta}\right)}{(\alpha_k-\alpha_l)(1-\e^{-2\alpha_k\delta})(\e^{2\alpha_l\delta}-1)}.
\end{equation}
It is clear that,
\begin{equation}\label{3.et1}
\frac{\alpha_k\alpha_l(1-\e^{-2(\alpha_k+\alpha_l)\delta})}{(\alpha_k+\alpha_l)(1-\e^{-2\alpha_k\delta})(1-\e^{-2\alpha_l\delta})}\leq C\frac{\alpha_k\alpha_l}{\alpha_k+\alpha_l}\leq C\left[\min(|k|,|l|)+1\right].
\end{equation}
As in the proof of (\ref{3A.10}), we have 
\begin{equation}\label{3.et2}
\left|\frac{\alpha_k\alpha_l(1-\e^{-2(\alpha_k-\alpha_l)\delta})}{(\alpha_k-\alpha_l)(1-\e^{-2\alpha_k\delta})(\e^{2\alpha_l\delta}-1)}\right|\leq \frac{C\alpha_k\alpha_l}{\max(|k|,|l|)}\leq C\left[\min(|k|,|l|)+1\right].
\end{equation}
Inequalities \eqref{3A.11} follows from \eqref{3.et1} and \eqref{3.et2}.
\subsection{Two fundamental estimates}\label{3.AP4}
In this part, we let $k>l\geq0$ and prove the following:
\begin{equation}\label{3A.13}
X_{k,l}:=\frac{(\alpha_k\alpha_l+kl+\lambda-1)(1-\e^{-2(\alpha_k+\alpha_l)\delta})}{(\alpha_k+\alpha_l)(1-\e^{-2\alpha_k\delta})(1-\e^{-2\alpha_l\delta})}=\frac{2kl}{k+l}+\mathcal{O}\left(\frac{1}{l+1}\right),
\end{equation}
\begin{equation}\label{3A.14}
Y_{k,l}:=\frac{(\alpha_k\alpha_l+kl+\lambda-1)(1-\e^{-2(\alpha_k-\alpha_l)\delta})}{(\alpha_k-\alpha_l)(1-\e^{-2\alpha_k\delta})(\e^{2\alpha_l\delta}-1)}\leq C\e^{-\delta l}.
\end{equation}
The computations are direct:
\begin{eqnarray*}
X_{k,l}-\frac{2kl}{k+l}&=&\frac{2kl}{(\alpha_k+\alpha_l)(1-\e^{-2\alpha_k\delta})(1-\e^{-2\alpha_l\delta})}-\frac{2kl}{k+l}+\mathcal{O}\left(\frac{1}{l+1}\right)
\\&=&2kl\frac{k+l-(\alpha_k+\alpha_l)(1-\e^{-2\alpha_k\delta})(1-\e^{-2\alpha_l\delta})}{(k+l)(\alpha_k+\alpha_l)(1-\e^{-2\alpha_k\delta})(1-\e^{-2\alpha_l\delta})}+\mathcal{O}\left(\frac{1}{l+1}\right)
\\&=&\frac{\mathcal{O}\left(k+k^2l\e^{-l\delta/2}\right)}{(k+l)(\alpha_k+\alpha_l)(1-\e^{-2\alpha_k\delta})(1-\e^{-2\alpha_l\delta})} +\mathcal{O}\left(\frac{1}{l+1}\right)
\\&=&\mathcal{O}\left(\frac{1}{l+1}\right).
\end{eqnarray*}
We now turn to \eqref{3A.14}.

If $\d\alpha_k\geq2\alpha_l$ (or equivalently, if $\alpha_k-\alpha_l\geq\frac{\alpha_k}{2}$), then
\begin{eqnarray*}
\frac{(\alpha_k\alpha_l+kl+\lambda-1)(1-\e^{-2(\alpha_k-\alpha_l)\delta})}{(\alpha_k-\alpha_l)(1-\e^{-2\alpha_k\delta})(\e^{2\alpha_l\delta}-1)}&\leq& C\frac{kl}{\alpha_k}\e^{-2\alpha_l\delta}\leq C\e^{-\delta l}.
\end{eqnarray*}
If $\d\alpha_k<2\alpha_l$, then
\begin{eqnarray*}
\frac{(\alpha_k\alpha_l+kl+\lambda-1)(1-\e^{-2(\alpha_k-\alpha_l)\delta})}{(\alpha_k-\alpha_l)(1-\e^{-2\alpha_k\delta})(\e^{2\alpha_l\delta}-1)}&\leq& Cl^2\e^{-2\alpha_l\delta} \leq C\e^{-\delta l}.
\end{eqnarray*}
\section{Proof of Proposition \ref{P3.1} and of Lemma \ref{L3.18}}\label{3.Apres}
\subsection{Proof of Proposition \ref{P3.1}}

The proof of 1) is direct by noting that if $u\in H^1(\dom,\mathbb{S}^1)$, then $\p_1u$ and $\p_2u$ are pointwise proportional and $\deg_{\p\Omega}(u)=\sum_i\deg_{\p\omega_i}(u)$,
\begin{eqnarray}\no
\ab_i(u,\dom)&=&\frac{1}{2\pi}\sum_{k=1,2}{(-1)^{k}\int_{ \dom}{(u\times\p_{k}u)\p_{3-k}V_i}}\\\no&=&\frac{1}{2\pi}\int_{\p \dom}{V_i\, u\times\p_\tau u\di\tau}=\deg_{\p\O}(u)-\sum_{j\neq i}\deg_{\p\omega_j}(u)=\deg_{\p\omega_i}(u).
\end{eqnarray}
Proof of 2). Since $V_i$ is locally constant on $\p \dom$, integrating by parts, 
\[
\int_\dom{v\times(\p_1u\,\p_2V_i-\p_2u\,\p_1V_i)\di x}=\int_\dom{u\times(\p_1v\,\p_2V_i-\p_2v\,\p_1V_i)\di x}.
\]
Then
\[
\begin{array}{ccl}
2\pi|\ab_i(u)-\ab_i(v)|&=&\d\Big|\int_\dom(u-v)\times\Big[(\p_1\,V_i\,\p_2u-\p_2V_i\,\p_1u)
\\&&\phantom{aaaaaaaqqqqqqqqqqa}+(\p_1\,V_i\,\p_2v-\p_2V_i\,\p_1v)\Big]\di x\Big|
\\
&\leq&\d\sqrt2\|u-v\|_{L^2(\dom)}\|V_i\|_{C^1(\dom)}(\|\n u\|_{L^2(\dom)}+\|\n v\|_{L^2(\dom)})
\\
&\leq&\d2\|u-v\|_{L^2(\dom)}\|V_i\|_{C^1(\dom)}[E_\v(u)^{1/2}+E_\v(v)^{1/2}]
\\
&\leq&\d4\|u-v\|_{L^2(\dom)}\|V_i\|_{C^1(\dom)}\Lambda^{1/2}.
\end{array}
\]
We prove assertion 3) by showing that $\dist(\ab_i(u_\v),\Z)=o(1)$. Using the first and the second assertion, we have  
\begin{eqnarray}\nonumber
\dist(\ab_i(u_\v),\Z)&\leq&\inf_{v\in E_0^\Lambda}{|\ab_i(u_\v)-\ab_i(v)|}
\\\label{3.7}
&\leq&\frac{2}{\pi}\|V_i\|_{C^1(\dom)}\Lambda^{1/2}\inf_{v\in E_0^\Lambda}{\|u_\v-v\|_{L^2(\dom)}}
\end{eqnarray}
where $\d E_0^\Lambda:=\left\{u\in H^1(\dom,\mathbb{S}^1)\text{ s.t. }\frac{1}{2}\int_\dom|\n u|^2\di x\leq\Lambda\right\}\neq\emptyset$.

Now, it suffices to show that $\inf_{v\in E_0^\Lambda}{\|u_\v-v\|_{L^2(\dom)}}\rightarrow0$. We argue by contradiction and we assume that there is an extraction $(\v_n)_n\downarrow0$ and $\delta>0$ s.t. for all $n$, $\d\inf_{v\in E_0^\Lambda}{\|u_{\v_n}-v\|_{L^2(\dom)}}>\delta$. 

We see that $(u_{\v_n})_n$ is bounded in $H^1$. Then, up to subsequence, $u_n$ converges to $u\in H^1(\dom,\R^2)$ weakly in $H^1$ and strongly in $L^4$. 

Since $\||u_{\v_n}|^2-1\|_{L^2(\dom)}\rightarrow0$, we have $u\in H^1(\dom,\mathbb{S}^1)$ and by weakly convergence, $\|\n u\|_{L^2(\dom)}^2\leq2\Lambda$. 

To conclude, we have $u\in E_0^\Lambda$ et $\|u_{\v_n}-u\|_{L^2}\rightarrow0$, which is a contradiction. 

\subsection{Proof of Lemma \ref{L3.18}}
1) We see easily that, with $z=\e^{\imath\theta}$, we have
\begin{equation}\label{3.egtrig}
\frac{\Psi_t(\overline{z})-\F_t(\overline{z})}{t}=\frac{(1-\varphi(\theta))(1-z^2)}{\left[z(1-t)-1\right]\left[z(1-t\varphi(\theta))-1\right]}\equiv\frac{A(\theta,t)}{B(\theta,t)}.
\end{equation}
The modulus of the RHS of (\ref{3.egtrig}) can be bounded by noting that 
\begin{description}
\item[ $\bullet$] there is some $m>0$ s.t. $|B(\theta,t)|\geq m$ for each $t$ and each $\theta$ s.t. $|\theta|>\delta/2$ mod $2\pi$;
\item[ $\bullet$] there is some $M>0$ s.t. $|A(\theta,t)|\leq M$ for each $t$ and each $\theta$ s.t. $|\theta|>\delta/2$ mod $2\pi$;
\item[ $\bullet$] if $|\theta|\leq\delta/2$ (modulo $2\pi$), then $\left(\Psi_t-\F_t\right)t^{-1}\equiv0$. %the modulus is big when $\theta$ is not far $2\pi\Z$, if $
\end{description}

2) This assertion is a standard expansion.

3) With a classical result relating regularity of $\Psi_t-\mathcal{F}_t$ to the asymptotic behaviour of its Fourier coefficients, we have
\[
|b_k(t)-c_k(t)|\leq \frac{2^{n+1}\pi\|\p_\theta^{^n}\left(\Psi_t-\F_t\right)\|_{L^\infty(\mathbb{S}^1)}}{t\left(1+|k|\right)^{n}}.
\]
Noting that, for $\d\p_\theta^{^n}\left(\Psi_t-\F_t\right)t^{-1}\equiv\frac{A_n(\theta,t)}{B_n(\theta,t)}$ 
\begin{description}
\item[ $\bullet$] there is some $m_n>0$ s.t. $|B_n(\theta,t)|\geq m_n$ for each $t$ and each $\theta$ s.t. $|\theta|>\delta/2$ mod $2\pi$;
\item[ $\bullet$] there is some $M_n>0$ s.t. $|A_n(\theta,t)|\leq M_n$ for each $t$ and each $\theta$ s.t. $|\theta|>\delta/2$ mod $2\pi$;
\item[ $\bullet$] if $|\theta|\leq\delta/2$ (modulo $2\pi$), then $\left(\Psi_t-\F_t\right)t^{-1}\equiv0$.
\end{description}
Thus the result follows.

\subsection{Proof of Lemma \ref{L3.simplification}}

The key argument to treat the energetic contribution of $D_\delta^\pm$ is the following lemma.
\begin{lem}\label{L3.19}
 \begin{enumerate}
\item $|\tilde{\psi_t}(h,\pm\delta)-1|=\mathcal{O}(t)$;
\item $|\p_h\tilde{\psi_t}(h,\pm\delta)|=\mathcal{O}(t|\ln t|)$.
\end{enumerate}

\end{lem}
\begin{proof}(of Lemma \ref{L3.19})\\
Using Lemma \ref{L3.18}, \eqref{3A.2} and (\ref{3A.7}), we have
\begin{eqnarray}\no
t^{-1}|\tilde{\psi_t}(h,\delta)-1|&%=&\left|-b_{-1}f_{-1}(h)+\sum_{k\neq-1}{b_k(t)f_k(h)\e^{-\imath[(k+1)\delta]}}\right|
%\\\no&
\leq &\left|-c_{-1}f_{-1}(h)+\sum_{k\neq-1}{c_kf_k(h)\e^{-\imath[(k+1)\delta]}}\right|
\\\no&&\phantom{aaa}+ \left|-(b_{-1}-c_{-1})f_{-1}(h)+\sum_{k\neq-1}{(b_k-c_k)f_k(h)\e^{-\imath(k+1)\delta}}\right|
\\\no&\leq&C(\delta)\left\{\left|\sum_{k\geq0}\left((1-t)\e^{-(1-h)-\imath\delta}\right)^k\right|+1\right\}=\mathcal{O}(1).%\mathcal{O}\left(\frac{1}{|1-(1-t)\e^{-(1-h)-\imath\delta}|}\right)\leq C(\delta)
\end{eqnarray}

We prove that $|\p_h\tilde{\psi_t}(h,\delta)|=\mathcal{O}(t|\ln t|)$. Using Lemma \ref{L3.18}, \eqref{3A.3} and (\ref{3A.8}),
\begin{eqnarray*}
t^{-1}|\p_h\tilde{\psi_t}(h,\delta)|&\leq&\left|-c_{-1}f'_{-1}+\sum_{k\neq-1}{c_kf'_k\e^{-\imath(k+1)\delta}}\right|
\\&&\phantom{aaaaa}+\left|-(b_{-1}-c_{-1})f'_{-1}+\sum_{k\neq-1}{(b_k-c_k)f'_k\e^{-\imath(k+1)\delta}}\right|
\\
&\leq&2\left|\sum_{k\geq0}{k\left[(1-t)\e^{-\imath\delta-(1-h)}\right]^k}\right|+\mathcal{O}(|\ln t|)= \mathcal{O}(|\ln t|).
\end{eqnarray*}
\end{proof}

Using (\ref{3.totala1}), (\ref{3.gpm-decomposition}) and Lemma \ref{L3.19}, we have (with the notation of section \ref{3.sectionPREUVE}) that
\[
M_\lambda(w_t,D_\delta)=R_\lambda(w_t)+o(t),
\]
where
\begin{eqnarray*}
R_\lambda(w_t)&=&\delta t^2\sum_{k\in\Z}{b_k^2\phi_k(f_k)}-2t^2\sum_{k\neq-1}{b_{-1}b_k\frac{\sin[(k+1)\delta]}{k+1}\int_{1-\delta}^1{[f_{-1}'f_k'-(k-\lambda+1)f_{-1}f_k}]}
\\\no
&&\phantom{aaaaaaaaaaaaaaaa}+2t^2\sum_{\substack{k, l\neq-1 \\ k-l > 0}}{b_kb_l\frac{\sin[(k-l)\delta]}{k-l}\int_{1-\delta}^1{[f_k'f_l'+(kl+\lambda-1)f_kf_l]}}.
\end{eqnarray*}
The proof of Lemma \ref{L3.19} is completed  provided we establish the following estimate:
\begin{equation}\label{3.diffpasgrande}
R_\lambda(w_t)\leq \delta-2\delta t+4t^2\sum_{\substack{k, l\geq0 \\ k-l > 0 }}{c_kc_l\frac{\sin[(k-l)\delta]}{k-l} \frac{kl}{k+l}}+o(t).
\end{equation}
The remaining part of this appendix is devoted to the proof of \eqref{3.diffpasgrande}.
\\{\bf We estimate the first term of $R_\lambda$:}

Using (\ref{3.est1}) and Lemma \ref{L3.18}, we have (with $C$ independent of $t$)
\begin{equation}\label{3.80}
\left| \sum_{k\in\Z}{b_k^2\phi_k(f_k)}- \sum_{k\in\Z}{c_k^2\phi_k(f_k)}\right|\leq C.
\end{equation}

With (\ref{3.est1}) and (\ref{3A.0}), we obtain
\begin{eqnarray}\label{3.DL}
\phi_k(f_k)&=&\alpha(1+\frac{2}{\e^{2\alpha\delta}-1})= |k|+\mathcal{O}\left(\frac{1}{|k|+1}\right)\text{ when }|k|\rightarrow\infty.
\end{eqnarray}
From  (\ref{3A.1}), (\ref{3A.3}) and (\ref{3.DL}), 
\begin{eqnarray}\no
t^2 \sum_{k\in\Z}{c_k^2\phi_k(f_k)}
&=&   t^2\phi_{-1}(f_{-1})+ t^2(t-2)^2\sum_{k\geq0}{(1-t)^{2k}\phi_{k}(f_k)}
\\\label{3.partie1}
&=& t^2(t-2)^2\sum_{k>0}{k(1-t)^{2k}}+o(t)=1-2t+o(t).
\end{eqnarray}
{\bf We estimate the second term of $R_\lambda$:}

Using Lemma \ref{L3.18}, (\ref{3A.10}) and (\ref{3A.11}), we have (with $C$ independent of $t$) 
\begin{equation}\no
\left|\sum_{k\neq-1}{(b_k-c_k)\frac{\sin[(k+1)\delta]}{k+1}\int_{1-\delta}^1{[f_{-1}'f_k'-(k-\lambda+1)f_{-1}f_k]}}\right|\leq  C.
\end{equation}
Since $b_{-1}(t)$ is bounded by a quantity independent of $t$, in the order to estimate the third term of the RHS of (\ref{3.totala}), we observe that there is $C$ independent of $t$ s.t.
\begin{eqnarray*}
\left|\sum_{k\geq0}{(1-t)^k\frac{\sin[(k+1)\delta]}{k+1}\int_{1-\delta}^1{[f_{-1}'f_k'-(k-\lambda+1)f_{-1}f_k]}}\right|&\leq& C\left( \sum_{k\geq 1}{\frac{(1-t)^k}{k}} +1\right)\\&=&C(|\ln t|+1).
\end{eqnarray*}
Finally, using Lemma \ref{L3.18}, (\ref{3.est3}) and (\ref{3.est4}), we have
\begin{equation}\label{3.part3}
\left|\sum_{k\neq-1}{b_k\frac{\sin[(k+1)\delta]}{k+1}\int_{1-\delta}^1{[f_{-1}'f_k'-(k-\lambda+1)f_{-1}f_k]}}\right|\leq C(|\ln t|+1).
\end{equation}
{\bf We estimate the last term of $R_\lambda$:}

First, we consider the case $k=-l>0$ (\emph{i.e.}, $f_k=f_l$). Using (\ref{3.est2}), $0\leq f_k\leq1$ and (\ref{3A.9}), we have (with $C$ independent of $t$)
\[
\left|\sum_{k>0}{b_kb_{-k}\frac{\sin2k\delta}{2k}\int_{1-\delta}^1{[f_k'^2+(-k^2+\lambda-1)f_k^2]}}\right|\leq C.%\neq C(t).
\]
It remains to estimate the last sum in $R_\lambda$, considered only over the indices $k$ and $l$ s.t. $|k|\neq|l|$. We start with 
\begin{eqnarray}\label{3.sum-dec}
&&\sum_{\substack{k, l\neq-1 \\ k-l > 0,k\neq-l }}{(b_kb_l-c_kc_l)\frac{\sin[(k-l)\delta]}{k-l}\int_{1-\delta}^1{[f_k'f_l'+(kl+\lambda-1)f_kf_l]}}
\\\no
&&=\sum_{\substack{k, l\neq-1 \\ k-l > 0,k\neq-l }}\left[(b_k-c_k)(b_l-c_l)+c_k(b_l-c_l)+c_l(b_k-c_k)\right]*\\\no&&\phantom{aaaaaaaaaa}*\frac{\sin[(k-l)\delta]}{k-l} \int_{1-\delta}^1{[f_k'f_l'+(kl+\lambda-1)f_kf_l]}.
\end{eqnarray}
By Assertion 3) of Lemma \ref{L3.18}, the first sum of the RHS of (\ref{3.sum-dec}) is easily bounded by a quantity independent of $t$. By (\ref{3A.10}), (\ref{3A.11}) and Lemma \ref{L3.18},
\begin{eqnarray*}
\Bigg|\sum_{\substack{k, l\neq-1 \\ k-l > 0,k\neq-l }}c_k(b_l-c_l)\frac{\sin[(k-l)\delta]}{k-l}&& \int_{1-\delta}^1{[f_k'f_l'+(kl+\lambda-1)f_kf_l]}\Bigg|
\\&&\leq C\sum_{\substack{k\geq0, l\neq-1 \\ k-l > 0,k\neq-l }}{\frac{ (1-t)^k|b_l-c_l||l|}{k-l}}+C.
%\\&\leq&C\sum_{\substack{n>0\\ l\in\Z  }}{\frac{ (1-t)^n|b_l-c_l||l|}{k-l}}+C.
\end{eqnarray*}
On the other hand (putting $n=k-l$),
\begin{eqnarray*}
\sum_{\substack{k\geq0, l\neq-1 \\ k-l > 0,k\neq-l }}{\frac{ (1-t)^k|b_l-c_l||l|}{k-l}}&\leq&\sum_{\substack{k>l\geq0 }}{\frac{ (1-t)^k|b_l-c_l|l}{k-l}}+\sum_{\substack{k\geq0, l\leq-1}}{\frac{ (1-t)^k|b_l-c_l||l|}{k+|l|}}
\\
&\leq&\sum_{l\geq0,n>0 }{\frac{ (1-t)^{n}}{n}|b_l-c_l|l}+\!\!\!\sum_{k>0, l\leq-1}{\frac{ (1-t)^k}{k}|b_l-c_l||l|}
\\&=& \mathcal{O}(|\ln t|).
\end{eqnarray*}
Similarly, we may prove that
\[
\left|\sum_{\substack{k, l\neq-1 \\ k-l > 0,k\neq-l }}c_l(b_k-c_k)\frac{\sin[(k-l)\delta]}{k-l} \int_{1-\delta}^1{[f_k'f_l'+(kl+\lambda-1)f_kf_l]}\right|=\mathcal{O}(|\ln t|).
\]
We have thus proved that
\[
\left|\sum_{\substack{k, l\neq-1 \\ k-l > 0,k\neq-l }}{(b_kb_l-c_kc_l)\frac{\sin[(k-l)\delta]}{k-l}\int_{1-\delta}^1{[f_k'f_l'+(kl+\lambda-1)f_kf_l]}}\right|=o(t^{-1}).
\]
To finish the proof, it suffices to obtain
\begin{eqnarray*}
\sum_{\substack{k, l\neq-1 \\ k-l > 0,k\neq-l }}c_kc_l\frac{\sin[(k-l)\delta]}{k-l}&&\int_{1-\delta}^1{[f_k'f_l'+(kl+\lambda-1)f_kf_l]}\\&&\phantom{aaaaaaaa}= 2\sum_{\substack{k, l\geq0 \\ k-l > 0 }}{c_kc_l\frac{\sin[(k-l)\delta]}{k-l} \frac{kl}{k+l}}+o(t^{-1}).
\end{eqnarray*}

Since $c_m=0$ for $m<-1$, it suffices to consider the case $k>l\geq0$. Under these hypotheses, we have by (\ref{3.est3}), (\ref{3.est4}), (\ref{3A.13}) and (\ref{3A.14}), 
\begin{eqnarray*}
\sum_{k>l\geq0 }{c_kc_l\frac{\sin[(k-l)\delta]}{k-l}\int_{1-\delta}^1{[f_k'f_l'+(kl+\lambda-1)f_kf_l]}}&=& 2\sum_{k>l\geq0 }{c_kc_l\frac{\sin[(k-l)\delta]}{k-l} \frac{kl}{k+l}}
\\&&+\mathcal{O}\left(\sum_{k> l\geq0}{\frac{c_kc_l|\sin[(k-l)\delta]|}{k-l}\frac{1}{l+1}}\right).
\end{eqnarray*}
We conclude by noting that
\[
\left|\sum_{k>l\geq0 }{c_kc_l\left|\frac{\sin[(k-l)\delta]}{(k-l)(l+1)}\right|}\right|\leq  C\left(1+\sum_{n>0}{\frac{(1-t)^n}{n}}\sum_{l>0}{\frac{(1-t)^{2l}}{l}}\right)\leq C(1+\ln^2 t).
\]
\section{Proof of Lemma {\ref{P3.bougedegmult}}}\label{3.bougerledeg}
\begin{lem}\label{L3.lemmebougedeg1}
Let $0<\eta,\delta<1$, there is
\begin{equation}\label{3.movedeg}
\begin{array}{cccc}
M_{\eta,\delta}:&D(0,1)&\rightarrow&\C\\&x&\mapsto&M_{\eta,\delta}(x)
\end{array}\text{ s.t.:}
\end{equation}
 \begin{enumerate}[i)]
 \item  $|M_{\eta,\delta}|=1$ on $\mathbb{S}^1$, $\deg_{\mathbb{S}^1}(M_{\eta,\delta})=1$,
 \item $\d\frac{1}{2}\int_{D(0,1)}|\n M_{\eta,\delta}|^2\leq\pi+\eta$,
\item $|M_{\eta,\delta}|\leq2$
\item if $|\theta|>\delta\text{ mod }2\pi$, then $M_{\eta,\delta}(\e^{\imath\theta})=1$. 
 \end{enumerate}
\end{lem}
\noindent {\bf Claim: }Taking $\overline{M_{\eta,\delta}}$ instead of $M_{\eta,\delta}$, we obtain the same conclusions replacing the assertion {\it i)} by $\deg_{\mathbb{S}^1}(\overline{M_{\eta,\delta}})=-1$.
\begin{proof}
As in section \ref{3.sectionPREUVE}, let $\varphi\in C^\infty(\R,\R)$ be s.t. 
\begin{itemize}\item $0\leq\varphi\leq1$, 
\item $\varphi$ is even and $2\pi$-periodic, 
\item $\varphi_{|(-\delta/2,\delta/2)}\equiv1$ and $\varphi_{|[-\pi,\pi[\setminus(-\delta,\delta)}\equiv0$.
\end{itemize}
For $0<t<\delta$, let $M_t=M$ be the unique solution of 
\begin{equation}\no
\left\{\begin{array}{cccl}
M(\e^{\imath\theta})&=&\d\frac{\e^{\imath\theta}-(1-t\varphi(\theta))}{\e^{\imath\theta}(1-t\varphi(\theta))-1}&\text{on }\p D(0,1)
\\
\Delta M&=&0&\text{in }D(0,1)
\end{array}\right..
\end{equation}
It follows easily that $M$ satisfies {\it i)}, {\it iii)} and {\it iv)}. We will prove that for $t$ small {\it ii)} holds.

Using \eqref{3.decdeM}, we have
\begin{equation}\label{3.decdeMbis}
\frac{\e^{\imath\theta}-(1-t\varphi(\theta))}{\e^{\imath\theta}(1-t\varphi(\theta))-1}=(1-tb_{-1}(t))+t\sum_{k\neq-1}{b_k(t)\e^{(k+1)\imath\theta}}.
\end{equation}
It is not difficult to see that
\begin{equation}\label{3.decdeMbistous}
M(r\e^{\imath\theta})=(1-tb_{-1}(t))+t\sum_{k\neq-1}{b_k(t)r^{|k+1|}\e^{(k+1)\imath\theta}}.
\end{equation}
From \eqref{3.decdeMbistous},
\begin{eqnarray}\no
\frac{1}{2}\int_{D(0,1)}|\n M|^2&=&t^2\int_0^{2\pi}\di\theta\int_0^1\di r\sum_{k\neq-1}{b_k^2(k+1)r^{2|k+1|-2}}
\\\no
&=&\pi t^2\sum_{k\geq0}{b_k^2(k+1)}+\pi t^2\sum_{k\leq-2}{|k+1|b_k^2}
\\\no%\label{3.decdanslasomme0}
&=&\pi t^2\sum_{k\geq0}{c_k^2(k+1)}+\mathcal{O}(t^2)\text{ (using Lemma \ref{L3.18})}
\\\no%\label{3.decdanslasomme1}
&=&\pi(2-t)^2 t^2\sum_{k\geq0}{(1-t)^{2k}(k+1)}+\mathcal{O}(t^2)\text{ (using Lemma \ref{L3.18})}
\\\no%\label{3.decdanslasomme2}
&=&\pi+\mathcal{O}(t^2)\text{ (using \eqref{3A.2} and \eqref{3A.3})}
\\\no
&\leq&\pi+\eta\text{ for $t$ small}.
\end{eqnarray}
We finish the proof taking, for $t$ small, $M_{\eta,\delta}=M_t$.
\end{proof}
\begin{lem}\label{L3.bougedegfinal}
Let $u\in\J$, $i\in\{0,...,N\}$ and $\v>0$. For all $\eta>0$, there is 
\[
u_\eta^\pm\in\J_{\deg(u,\dom)\pm{\bf e}_i}
\]
s.t.
\begin{equation}\label{3.bougedegeta}
E_\v(u_\eta^\pm)\leq E_\v(u)+\pi+\eta
\end{equation}
%\begin{equation}\label{3.normeinf}
%\|u^\pm_\eta\|_{L^\infty(\dom)}\leq\max(2+2\eta^2,\|u\|_{L^\infty(\dom)})
%\end{equation}
and
\begin{equation}\label{3.convennormeL2}
\|u-u_\eta^\pm\|_{L^2(\dom)}=o_\eta(1),\:o_\eta(1)\underset{\eta\goto0}{\goto}0.
\end{equation}
\end{lem}
\begin{proof}
We prove that for $i=0$, there is $u^+_\eta\in\J_{\deg(u,\dom)+{\bf e}_i}$ satisfying \eqref{3.bougedegeta} and \eqref{3.convennormeL2}. In the other cases the proof is similar.

Using the density of $C^0(\overline{\dom},\C)\cap\J$ in $\J$ for the $H^1$-norm, we may assume $u\in C^0(\overline{\dom},\C)\cap\J$.

It suffices to prove the result for $0<\eta<\min\{10^{-3},\v^2\}$.

Let $x^0\in\p\O$ and $V_\eta$ be an open regular set of $\dom$ s.t. :
\begin{enumerate}[$\bullet$]
\item $\p V_\eta\cap\p\dom\neq\emptyset$, $\d|V_\eta|\leq\eta^2$, 
\item $x^0$ is an interior point of $\p\O\cap\p V_\eta$,
\item $V_\eta$ is simply connected,
\item $|u|^2\leq\d1+\eta^2$ in $V_\eta$,
\item $\|\n u\|_{L^2(V_\eta)}\leq \eta^2$.
\end{enumerate}
Using the Carathéodory's theorem, there is 
\[
\Phi: \overline{V_\eta}\rightarrow \overline{D(0,1)},
\]
a homeomorphism s.t. $\Phi_{| V_\eta}:V_\eta\goto D(0,1)$ is a conformal mapping.

Without loss of generality, we may assume that $\Phi(x^0)=1$. Let $\delta>0$ be s.t. for $|\theta|\leq\delta$ we have $\Phi^{-1}(\e^{\imath\theta})\in\p V_\eta\cap\p\O$.

Let $N_\eta\in\J$ be defined by
\[
N_\eta(x)=\left\{\begin{array}{cl}1&\text{if }x\in\dom\setminus V_\eta\\ M_{\eta^2,\delta}(\Phi(x))&\text{otherwise}
\end{array}\right..
\]
Here, $M_{\eta^2,\delta}$ is defined by Lemma \ref{L3.lemmebougedeg1}. Using the conformal invariance of the Dirichlet functional, we have
\begin{equation}\label{3.invconf}
\frac{1}{2}\int_{V_\eta}{|\n N_\eta|^2}=\frac{1}{2}\int_{D(0,1)}{|\n M_{\eta^2,\delta}|^2}\leq \pi+\eta^2.
\end{equation}

It is not difficult to see that $u^+_\eta:=uN_\eta\in\J_{\deg(u,\dom)+{\bf e}_0}$. Since $|N_{\eta}|\leq2$ and $\|N_\eta-1\|_{L^2(\dom)}=o_\eta(1)$, using the Dominated convergence theorem, we may prove that $uN_\eta\rightarrow u$ in $L^2(\dom)$ when $\eta\rightarrow0$. It follows that \eqref{3.convennormeL2} holds.

From \eqref{3.invconf} and using the following formula,
\[
|\n (uv)|^2=|v|^2|\n u|^2+|u|^2|\n v|^2+2\sum_{j=1,2}(v\p_ju)\cdot(u\p_j v)
\]
we obtain
\begin{eqnarray}\no
\frac{1}{2}\int_{V_\eta}{|\n u^+_\eta|^2}&=&\frac{1}{2}\int_{V_\eta}{\left\{|N_\eta|^2|\n u|^2+|u|^2|\n N_\eta|^2+2\sum_{j=1,2}(N_\eta\p_ju)\cdot(u\p_j N_\eta)\right\}}
\\\no
&\leq&(1+\eta^2)(\pi+\eta^2)+2\|\n u\|_{L^2(V_\eta)}^2+4\sqrt{1+\eta^2}\|\n u\|_{L^2(V_\eta)}\|\n N_\eta\|_{L^2(V_\eta)}
\\\label{3.dirdansVeta}
&\leq&\pi+\frac{\eta}{2}.
\end{eqnarray}
Furthermore, we have
\begin{equation}\label{3.restedansVeta}
\frac{1}{4\v^2}\int_{V_\eta}{(1-|u_\eta^+|^2)^2}\leq\frac{\eta^2}{4\v^2}\leq\frac{\eta}{2}.
\end{equation}
From \eqref{3.dirdansVeta} and \eqref{3.restedansVeta}, it follows
\[
E_\v(u^+_\eta,\dom)=E_\v(u,\dom\setminus V_\eta)+E_\v(u^+_\eta, V_\eta)\leq E_\v(u,\dom)+\pi+\eta.
\]
The previous inequality completes the proof.
\end{proof}
We may now prove Lemma \ref{P3.bougedegmult}. For the convenience of the reader, we recall the statement of the lemma.

\begin{lemimnoi}
Let $u\in\J$, $\v>0$ and ${\bf \delta}=(\delta_1,...,\delta_N,\delta_0)\in\Z^{N+1}$. For all $\eta>0$, there is $u_\eta^\delta\in\J_{\deg(u,\dom)+\delta}$ s.t.
\begin{equation}\tag{\ref{3.bougedegetamul}}
E_\v(u_\eta^\delta)\leq E_\v(u)+\pi\sum_{i\in\{0,...,N\}}{|\delta_i|}+\eta
\end{equation}

and
\begin{equation}\tag{\ref{3.convennormeL2mul}}
\|u-u_\eta^\delta\|_{L^2(\dom)}=o_\eta(1),\:o_\eta(1)\underset{\eta\goto0}{\goto}0.
\end{equation}
\end{lemimnoi}
\begin{proof}

As in the previous lemma, it suffices to prove the proposition for  $0<\eta<\min\{10^{-3},\v^2\}$ and $u\in C^0(\overline{\dom},\C)\cap\J$. 

We construct $u^\delta_\eta$ in $\d\ell_1=\sum_{i\in\{0,...,N\}}|\delta_i|$ steps. If $\ell_1=0$ (which is equivalent at $\delta={\bf 0}_{\Z^{N+1}}$) then, taking $u^\delta_\eta=u$, \eqref{3.bougedegetamul} and \eqref{3.convennormeL2mul} hold.

Assume $\ell_1\neq0$. Let $\Gamma=\{i\in\{0,...,N\}\,|\,\delta_i\neq0\}\neq\emptyset$, $L=\text{Card}\,\Gamma$ and $\mu=\d\frac{\eta}{\ell_1}$. We enumerate the elements of $\Gamma$ in $(i_n)_{n\in\N_L}$ s.t. for $n\in\N_{L-1}$ we have $i_n<i_{n+1}$.

Let $\sigma$ be the sign function \emph{i.e.} for $x\in\R^*$, $\d\sigma(x)=\frac{x}{|x|}$.

For $n\in\N_{L}$ and $l\in\N_{|\delta_{i_n}|}$, we construct 
\[
v^l_n\in\J_{\deg(v^{l-1}_n,\dom)+\sigma(\delta_i){\bf e }_{i_n}}
\]
s.t
\[
\begin{array}{c}
 v_0^0=u,\:\:v_{n}^0=v_{n-1}^{|\delta_{i_{n-1}}|} \,\text{ with for $n=1$, $\delta_{i_0}=0$},
 \\
 v_{n}^{l+1}=\left\{\begin{array}{cl}(v^l_n)_\mu^+&\text{if }\delta_{i_n}>0\\(v^l_n)_\mu^{-}&\text{if }\delta_{i_n}<0\end{array}\right.,\:0\leq l<|\delta_{i_{n}}|
\end{array}.
\]
Here, $(v^l_n)_\mu^\pm$ stands for $u^\pm_\mu$ defined by Lemma \ref{L3.bougedegfinal} taking $u=v^l_n$ and $\eta=\mu$.

It is clear that $v^l_n$ is well defined and  that for $n\in\N_L$, $v_n:=v^{|\delta_{i_n}|}_n\in\J_{\deg(v_{n-1},\dom)+\delta_{i_n}{\bf e }_{i_n}}$ with $v_0=u$.

Therefore, using \eqref{3.bougedegeta}, we have for $n\in\N_L$,
\[
v_n\in\J_{\deg(u,\dom)+\sum_{k\in\N_n}\delta_{i_k}{\bf e }_{i_k}},\: E_\v(v_n)\leq E_\v(u)+ \left(\pi+\mu\right)\sum_{k\in\N_n}|\delta_{i_k}|.
\]
Taking $n=L$, we obtain that
\[
u_\eta^\delta=v_L\in\J_{\deg(u,\dom)+\delta},\: E_\v(u_\eta^\delta)\leq E_\v(u)+\pi\sum_{i\in\{0,...,N\}}{|\delta_i|}+\eta.
\] 
Furthermore, $u_\eta^\delta$ is obtained from $u$ multiplying  by $\ell_1$ factors $N_l$, $l\in\N_{\ell_1}$. Each $N_l$ is bounded by $2$ and converges to $1$ in $L^2$-norm (when $\eta\goto0$). Using the Dominated convergence theorem, we may prove that $u_\eta^\delta$ satisfies \eqref{3.convennormeL2mul}.

\end{proof}
%\end{subappendices}
\noindent\emph{Acknowledgements.} The author would like to express his gratitude to Professor Petru Mironescu for suggesting him to study the problem treated in this paper and for his useful remarks. 

\bibliography{biblio}
\end{document}